\documentclass[oneside,10pt]{amsart}
\usepackage[left=24mm,right=36mm]{geometry} 
\renewcommand{\bar}{\overline}
\usepackage{amsmath,amsthm,amscd,euscript,amssymb,graphicx,enumerate}

\theoremstyle{plain}
\newtheorem{theorem}{Theorem}[section]
\newtheorem{lemma}[theorem]{Lemma}

\newtheorem{proposition}[theorem]{Proposition}
\newtheorem{corollary}[theorem]{Corollary}

\theoremstyle{remark}
\newtheorem{remark}[theorem]{Remark}

\theoremstyle{definition}
\newtheorem{definition}[theorem]{Definition}
\newtheorem{example}[theorem]{Example}
\newtheorem{problem}{Problem}

\usepackage[markup=underlined,todonotes={textsize=footnotesize}]{changes}
\definechangesauthor[color=red]{A1}
\definechangesauthor[color=blue]{A2}
\definechangesauthor[color=green]{OK}
\definechangesauthor[color=cyan]{A3}

\def\st{\,\,\big|\,\,}

\def\<{\langle}
\def\>{\rangle}

\let\ge=\geqslant
\let\le=\leqslant
\let\emptyset=\varnothing

\def\calD{{\mathcal D}}
\def\calF{{\mathcal F}}
\def\calR{{\mathcal R}}
\def\calS{{\mathcal S}}

\renewcommand{\bar}{\overline}

\def \r{\mathbb R}
\def \c{\mathbb C}
\def \h{\mathbb H}

\def \d{\mathbb D}
\def \q{\mathbb Q}
\def \z{\mathbb Z}

\def\al{{\alpha}}
\def\be{{\beta}}
\def\ga{{\gamma}}
\def\de{{\delta}}
\def\ve{{\varepsilon}}

\def\Ga{{\Gamma}}
\def\De{{\Delta}}
\def\La{{\Lambda}}

\DeclareMathOperator{\Br}{Br}

\DeclareMathOperator{\conj}{conj}

\DeclareMathOperator{\Isom}{Isom}
\DeclareMathOperator{\Mob}{Mob}
\DeclareMathOperator{\PSL}{PSL}

\let\Im=\undefined \DeclareMathOperator{\Im}{Im}

\def\An{A_0}
\def\Ap{A_{+}}
\def\Am{A_{-}}
\def\Apm{A_{\pm}}

\def\sn{s_0}
\def\sp{s_{+}}
\def\sm{s_{-}}
\def\spm{s_{\pm}}
\def\tp{T_{+}}
\def\tm{T_{-}}
\def\tpm{T_{\pm}}

\def\trio{\De}
\def\trioB{\trio_B}

\def\Oeven{O_{\text{even}}}
\def\Oodd{O_{\text{odd}}}
\def\reven{r_{\text{even}}}
\def\rodd{r_{\text{odd}}}
\def\Deven{D_{\text{even}}}
\def\Dodd{D_{\text{odd}}}
\def\Seven{S_{\text{even}}}
\def\Sodd{S_{\text{odd}}}

\title[Farey Bryophylla]{Farey Bryophylla}
\author{Oleg Karpenkov, Anna Pratoussevitch}
\date{\today}

\usepackage{color}

\newcommand{\OK}[1]{{\color{blue}#1}}

\begin{document}
\maketitle
\input{epsf}

\begin{abstract}
The construction of the Farey tessellation in the hyperbolic plane starts with a finitely generated group of symmetries of an ideal triangle, i.e.\ a triangle with all vertices on the boundary.
It induces a remarkable fractal structure on the boundary of the hyperbolic plane, encoding every element by the continued fraction related to the structure of the tessellation.
The problem of finding a generalisation of this construction to the higher dimensional hyperbolic spaces has remained open for many years. 
In this paper we make the first steps towards a generalisation in the three-dimensional case.
We introduce conformal bryophylla,
a class of subsets of the boundary of the hyperbolic 3-space which possess fractal properties similar to the Farey tessellation.
We classify all conformal bryophylla and study the properties of their limiting sets.
\end{abstract}

\tableofcontents


\section*{Introduction}
\label{section-intro}

\noindent
In this paper we introduce a new generalisation of
the Farey fractal on the boundary of $\h^2$ to the boundary of $\h^3$ and study its basic properties.

There is a well-studied connection between continued fractions, the modular group and the Farey tessellation.
The modular group~$\PSL(2,\z)$ consists of M\"obius transformations
\[
  z\mapsto\frac{\al z+\be}{\ga z+\de},\quad 
  \al,\be,\ga,\de\in\z
\]
and acts on the hyperbolic plane~$\h^2$
(in the upper half-plane model).
The Farey tessellation is a decomposition of the hyperbolic plane~$\h^2$ into the images of the ideal triangle with vertices~$0,1,\infty$ under the action of the modular group, see Figure~\ref{Farey-tes}.
There is a correspondence between the continued fraction of~$x\in\r$ and the way in which a hyperbolic geodesic from some point on the positive imaginary axis to~$x$ cuts across the Farey tessellation
of the hyperbolic plane~$\h^2$ as described in~\cite{S85, S15}.
There are many natural generalisations of this correspondence to more general classes of matrix semi-groups and the corresponding versions of continued fractions and tessellations, such as complex continued fractions (see~\cite{Ho, FKST} and references therein), continued fractions on the Heisenberg group~\cite{LV} etc.

Our approach will be to replace the modular group with a semi-group~$S$ of M\"obius transformations and the boundary Farey fractal with the images of a triangle in~$\c$ under the action of~$S$.
Two-generator semi-groups of real M\"obius transformations were considered in~\cite{Wa},
where $S$-continued fractions were defined as the expansions corresponding to a semi-group~$S$.
In this paper we will introduce and classify a special class of two-generator semi-groups
of complex M\"obius transformations such that the images of a certain triangle under the semi-group form a conformal self-similar fractal set which we call a {\sl bryophyllum}.
The reasons for the choice of the name are explained in Example~\ref{example-L}.
We will focus on the fractal structure of the absolute of these semi-groups.
%
%
Our paper makes the first steps to the solution of the following general question (Problem~35 in~\cite{Karpenkov-book}):
{\it Find a natural generalisation of the Farey tessellation in higher dimensional hyperbolic geometry.} 

This paper is organised as follows.
In Section~\ref{section-background} we recall the classical construction of the Farey tessellation and some general notions of M\"obius geometry.
In Section~\ref{section-conformal-bryophylla} we introduce a general definition of conformal bryophylla and describe the subset of Farey bryophylla.
We show that any bryophyllum is conformally equivalent to precisely one bryophyllum from a particular two-parameter family of canonical conformal bryophylla (Theorem~\ref{theorem-clasification}).
Finally we provide a complete classification of all Farey canonical bryophylla up to conformal transformations (Theorem~\ref{theorem-Farey-region}).
In Section~\ref{Section-on-convergency} we consider a sequence of nested iterates of bryophyllum mappings and state that every such nested sequence converges to a point (Theorem~\ref{convergence-triangles}).
Furthermore we discuss one-dimensional continuous parametrisations of the limiting sets of Farey bryophylla by the Farey coordinate.
In particular we discuss the continued fraction Gauss map and the Gauss measure on Farey bryophylla induced by the Farey coordinate.
Finally in Section~\ref{Proofs of the main results} we provide proofs of the main results.
We conclude with a short Section~\ref{Open questions for further study} outlining several interesting open questions which are beyond the scope of this paper.

\section{Background}
\label{section-background} 

Let us briefly outline the classical construction of the Farey tessellation and the induced action on the boundary of the hyperbolic space.

\subsection{Classical Farey Tessellation}

A triangle in the hyperbolic plane~$\h^2$ is called {\it ideal} if it is formed by three points on the boundary and three geodesics joining them.
It turns out that all ideal triangles are conformally equivalent to each other. 
So one can choose any three points on the boundary as vertices of the triangle.
When working in the upper half-plane model of~$\h^2$,
a convenient choice of vertices is~$0$, $1$, and~$\infty$.
We denote the geodesic ideal triangle with vertices~$0$, $1$, and~$\infty$ by~$\mathcal T$.

Consider the group of symmetries generated by the hyperbolic reflections with respect to the sides of the triangle~$\mathcal T$.
This group acts freely on the hyperbolic plane.
The triangle~$\mathcal T$ is a fundamental domain for this action.
The images of~$\mathcal T$ completely cover the hyperbolic plane without intersections.
Such decomposition of the hyperbolic plane into ideal triangles is called the {\it Farey tessellation} (see Figure~\ref{Farey-tes}).  

\begin{figure}[t]
\centering
\includegraphics[width=13cm]{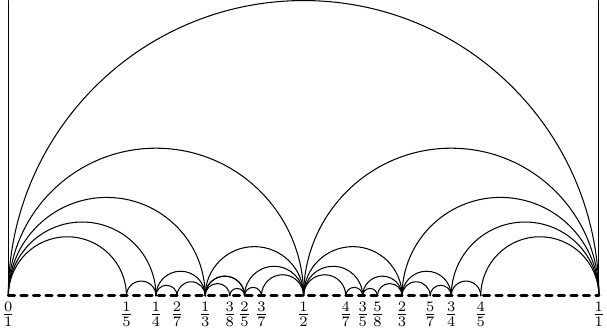}
\caption{Farey tessellation.}
\label{Farey-tes}
\end{figure}

\subsection{Induced Action on the Boundary}
\label{Induced action on the boundary}

The Farey tessellation provides a remarkable structure on the boundary $\partial\h^2=\r\cup\{\infty\}$;
the boundary can be seen in Figure~\ref{Farey-tes} as a dashed horizontal line.
It turns out that the set of all vertices of images of the triangle~$\mathcal T$ on the extended real line coincides with the set of all rational numbers. 
Two rational numbers $p/q$ and $r/s$ are connected by an edge of the tessellation if and only if $|ps-rq|=1$.
Here we assume that pairs of integers~$(p,q)$ and~$(r,s)$ are relatively prime.

Let $B$ be the set of all segments in~$\partial \h^2$ whose vertices are the vertices of the same triangle in the Farey tessellation of~$\h^2$.
Note that none of the segments in~$B$ has integer points in the interior, and that $B$ is invariant under integer shifts.
For these two reasons, we restrict ourselves to the unit segment $I=[0,1]$ and consider the corresponding set $B_I\subset B$ 
consisting of the segments in~$B$
that are completely contained in~$I$.

It turns out that the set~$B_I$ admits a nested structure of a binary tree that is perfectly described by the following two M\"obius maps:
$$
  \tp:x\to\frac{x}{x+1}
  \qquad\hbox{and}\qquad
  \tm:x\to\frac{1}{2-x}.
$$
These maps correspond to two parabolic elements
fixing the points $0$ and $1$ respectively.
They both preserve the Farey tessellation.

Let us consider the semi-group generated by~$\tp$ and~$\tm$.
The images of~$I$ under this semi-group
give all the elements of~$B_I$.
Additionally, different elements of the semi-group produce different elements of~$B_I$.

%

This structure is classically studied in the geometry of numbers~\cite{HW}, hyperbolic geometry~\cite{S85,S15}, Frieze patterns and cluster algebras~\cite{SVZ}, etc.

In this paper we are concerned with the following question:
{\it What would be a generalisation of the triple~$(I,T_+,T_-)$ in~$\h^n$ for $n>2$?}

\subsection{On Hyperbolic and M\"obius Geometry}
\label{section-mobius-geometry}

Let us describe the connection between hyperbolic and M\"obius geometry, for more details see~\cite{FMS, Bea}.
The upper half-space model of the real hyperbolic space
$\h_{\r}^{n+1}=\r^n\times\r_+$
has the metric $ds^2/x_{n+1}^2$.
The boundary of~$\h_{\r}^{n+1}$ is the extended Euclidean space
$\partial\h_{\r}^{n+1}=\r^n\cup\{\infty\}$.
The group of isometries of~$\h_{\r}^{n+1}$
is generated by all reflections in planes that are orthogonal to~$\partial\h_{\r}^{n+1}$
and all inversions in hemi-spheres that are orthogonal to~$\partial\h_{\r}^{n+1}$.
The M\"obius group~$\Mob(\r^n\cup\{\infty\})$ of the extended Euclidean space is generated by all reflections in planes and all inversions in hemi-spheres.
The restriction of isometries of~$\h_{\r}^{n+1}$ to the boundary~$\partial\h_{\r}^{n+1}$ induces an isomorphism
$\Isom(\h_{\r}^{n+1})\to\Mob(\r^n\cup\{\infty\})$.





\subsection{Higher Dimensional Situation}

While there have been many attempts to generalise the example $(I,\tp,\tm)$ to higher dimensions, none of them were entirely successful due to the  topological and geometric obstacles that we would like to discuss below. 

\vspace{2mm}

{\noindent \underline{Topological obstacles:}}
In the one-dimensional case the action of the group
of symmetries generated by the hyperbolic reflections with respect to the sides of the triangle~$\mathcal T$ can be identified with the action of the subgroup of integer matrices with determinant one.
We can interpret a fraction~$p/q$ as a projective vector~$(p:q)$. 
Such approach has a straightforward generalisation to the higher dimensional projective planes consisting of points $(p_0:\cdots:p_n)$.
Various semi-groups of integer matrices can be used,
such as Jacobi-Perron continued fractions and subtractive algorithms, see~\cite{Sch} and~\cite{Karpenkov-book}.
The downside of this approach is that it leads to projective rather than hyperbolic spaces, see~\cite{Wa,GR,FKST}.
In particular, the boundary of~$\h^n$ is the space~$\r^{n-1}$ compactified by a single point, 
while $\r P^{n-1}$ is compactified by~$\r P^{n-2}$.
As a result, all such generalisations are not within the framework of hyperbolic geometry.

\vspace{2mm}

{\noindent \underline{Geometric obstacles:}}
Since our aim is to get a generalisation of the Farey tessellation in the hyperbolic setting, we would like to stay within the group of hyperbolic reflections in hemi-spheres in~$\h^n$.
The boundaries of these hemi-spheres are $(n-1)$-dimensional spheres.
In the case~$n=2$, spheres of dimension~$(n-1)$ are just segments and we have $I=T_+(I)\cup T_-(I)$
with $T_+(I)\cap T_-(I)=\{1/2\}$.
However in the case~$n\ge2$, spheres of dimension~$(n-1)$ do not fit nicely together without interior intersections.
Therefore, they do not provide a tessellation of the boundary of the hyperbolic space. 

\vspace{2mm}
In this paper we study the case of certain remarkable subsets in~$\partial\h^3$  which we call {\it conformal bryophylla}.
A conformal bryophyllum is equipped with a pair of conformal mappings that inherit the main fractal properties of the classical Farey tessellation.
We provide a conformal classifications of 
conformal bryophylla and discuss their limiting invariant sets. 

\section{Bryophylla}
\label{section-conformal-bryophylla}

We start with the definitions of conformal bryo\-phil\-la and canonical bryo\-phil\-la in Subsections~\ref{subsection-conf-bryo} and~\ref{subsection-canon-bryo}.
In Subsection~\ref{subsection-classification} we state our main classification result for conformal bryophilla
that says that every conformal bryophyllum is conformally equivalent to a unique canonical bryophyllum.
Further in Subsection~\ref{subsection-examples} we illustrate the definitions with several examples.
Affine canonical bryophylla generated by affine transformations are discussed in Subsection~\ref{subsection-affine}.
For the edges of a bryophyllum to bound a curvilinear triangle and for the images of this triangle under~$\tpm$ to have similar inclusion properties to the classical Farey tessellation, we have to impose additional conditions.
We introduce the class of bryophylla satisfying these conditions in Subsection~\ref{subsection-farey}
and call them Farey bryophylla.
We conclude this section with the discussion of the structure of the configuration space of canonical bryophylla in Subsection~\ref{subsection-config}.

\subsection{Conformal Bryophylla}
\label{subsection-conf-bryo}

For simplicity we use the standard complex coordinates on the complex plane.
We will use the term {\it circline\/} for a subset that can be either a circle or a straight line.


\begin{definition}
\label{def-conf-bryo}
A {\sl $($conformal$)$ bryophyllum} $(\An,\Ap,\Am,\sp,\sm,\sn,\tp,\tm)$
is a collection of
\begin{enumerate}[$\bullet$]
\item
{\bf Vertices:} Three distinct points $\An$, $\Ap$ and~$\Am$;
\vspace{1mm}
\item
{\bf Edges:} Three arcs/segments, $\sn$ between $\Ap$ and $\Am$, $\sp$ between $\Ap$ and $\An$, and $\sm$ between $\Am$ and $\An$, such that the circlines determined by~$\sp$ and~$\sm$ do not coincide and $\sp\cap \sm=\{\An\}$;

%
\vspace{1mm}
\item
{\bf Maps:} Two conformal maps $\tp$ and $\tm$
that generate the fractal, satisfying the following conditions
\begin{enumerate}[$\circ$]
\item
$\tp(\Ap)=\An$, \quad  $\tp(\Am)=\Ap$, \quad $\tp(\sn)=\sp$;
\item
$\tm(\Am)=\An$,  \quad $\tm(\Ap)=\Am$, 
 \quad $\tm(\sn)=\sm$;
\item $\tp(\An)\in\sn$, 
 \quad $\tm(\An)\in\sn$;
\item
$\tp(\sm)$ and~$\tm(\sp)$ are both contained in the circle determined by the arc~$\sn$.
\end{enumerate}
\end{enumerate}
\end{definition}

\noindent
We discus several examples of conformal bryophylla later in Subsection~\ref{subsection-examples} but first we need to introduce some notation.

\vspace{2mm}
\noindent
Let us note the following general property of conformal bryophylla:

\begin{proposition}
\label{angles}
In a conformal bryophyllum $(\An,\Ap,\Am,\sp,\sm,\sn,\tp,\tm)$,
the angles between the following edges coincide:
\[\angle(\sn,\sp)=\angle(\sn,\sm).\]
\end{proposition}

\begin{proof}
The conformal map~$\tp$ maps the edges $\sn$ and $\sm$ to sub-arcs of $\sp$ and $\sn$,
hence the angles between these pairs of edges coincide, $\angle(\sn,\sp)=\angle(\sn,\sm)$.
\end{proof}

\subsection{Canonical Bryophylla}
\label{subsection-canon-bryo}

We will show in Section~\ref{subsection-classification} that any conformal bryophyllum can be transformed to a bryophyllum in a particular family of bryophylla called canonical.
Let us give the definition (see Figure~\ref{4mapping-123} for an example): 



\begin{definition}
\label{def-canon-bryo}
Let~$\varphi\in(0,\pi)$, $\psi\in[0,\pi]$ and $\varphi+\psi\ne\pi$.
A {\it canonical\/} bryophyllum $\Br_{\varphi,\psi}$
is the conformal bryophyllum $(\An,\Ap,\Am,\sp,\sm,\sn,\tp,\tm)$ with the following data:

\begin{itemize}
\item
{\bf Vertices:} $\An=1$, \quad  $\Ap=e^{i\varphi}$, \quad  $\Am=e^{-i\varphi}$.
\item
{\bf Edges:} The edges~$\spm$ are the straight line segments $s_{\pm}=[\Apm,\An]$ between $\Apm$ and $\An$.
If $\psi\ne\varphi/2$ then the edge~$\sn$ is the arc between~$\Ap$ and~$\Am$ of the circle through the points~$\Ap,\Am,re^{i\psi}$, where
\[r=\frac{\cos(\varphi/2+\psi)}{\cos(\varphi/2)}.\]
If $\psi=\varphi/2$ then the edge~$\sn$ is the straight line segment between the points $\Ap$ and~$\Am$.
\item
{\bf Maps:} The conformal maps $\tpm$ are given by
\[
  \tpm(z)
  =\frac
    {(re^{\pm i\psi}-\cos(\varphi/2)\cdot e^{\pm i\varphi/2})z+(\cos(\varphi/2)\cdot e^{\pm i\varphi/2}-r\cos(\varphi)e^{\pm i\psi})}
    {(r\cos(\varphi/2)\cdot e^{\pm i(\psi-\varphi/2)}-\cos(\varphi))z+(1-r\cos(\varphi/2)\cdot e^{\pm i(\psi-\varphi/2)})}.
\]
\end{itemize}
\end{definition}

\begin{remark}
\label{rem-canon-bryo}
Consider a canonical bryophyllum $Br_{\varphi,\psi}=(\An,\Ap,\Am,\sp,\sm,\sn,\tp,\tm)$.
Let $O$ be the origin.
Then
\begin{enumerate}[$\bullet$]
\item
The angle between~$\sp$ and~$\sm$ is~$\angle(\sp,\sm)=\pi-\varphi$.
\item
If $\varphi/2+\psi<\pi/2$ then $r=|O\tp(\An)|>0$, $\varphi=\angle\An O\Ap$, $\psi=\angle\An O\tp(\An)$.
\item
If $\varphi/2+\psi>\pi/2$ then $r=-|O\tp(\An)|<0$,
$\varphi=\angle\An O\Ap\pm\pi$, $\psi=\angle\An O\tp(\An)\pm\pi$.
\end{enumerate}
\end{remark}

\subsection{Classification of Conformal Bryophylla}
\label{subsection-classification}

We start with the definition of conformal equivalence of bryophylla.

\begin{definition}
Two conformal bryophylla are said to be {\it conformally equivalent\/} if there exists a conformal map~$T$ sending vertices, edges, and maps  of one bryophyllum
to the corresponding vertices, edges, and maps of the other one.
\end{definition}

\bigskip\noindent
In fact, every conformal bryophyllum is conformally equivalent to exactly one canonical conformal bryophyllum, so we have the following statement:

\begin{theorem}\label{theorem-clasification}
{\bf(Conformal classification of bryophylla)}

\begin{enumerate}[(1)]
\item
Any conformal bryophyllum $(\An,\Ap,\Am,\sp,\sm,\sn,\tp,\tm)$ is conformally equivalent to some canonical bryophyllum $\Br_{\varphi,\psi}$.
\item
Two canonical bryophylla $\Br_{\varphi_1,\psi_1}$ and $\Br_{\varphi_2,\psi_2}$ are conformally equivalent if and only if
$\varphi_1=\varphi_2$ and $\psi_1=\psi_2$.
\end{enumerate}
\end{theorem}

\bigskip\noindent
We will prove this theorem later in Subsections~\ref{theorem-clasification-1-proof}
and~\ref{theorem-clasification-2-proof}.


\subsection{Some Examples of Canonical Bryophylla}
\label{subsection-examples}

Let us consider some examples of canonical bryo\-phylla.

\begin{example}\label{exam2}
\label{example-L}
First we consider the canonical bryophyllum $B=\Br_{2\pi/3,\pi/6}$,
see Figure~\ref{fig-classic-bryo}.
The first 8 iteration steps of the arcs~$\spm$ and~$\sn$ of $\Br_{2\pi/3,\pi/6}$ with respect to the maps~$\tpm$ are shown on the left of Figure~\ref{fig-classic-bryo}.
This bryophyllum is very special as it is the only canonical bryophyllum with both
$\angle(\sp,\sn)=\angle(\sm,\sn)=0$ and $\tp(\An)=\tm(\An)$.
Therefore there is an additional symmetry:
$\tp(B)$ can be obtained from $\tm(B)$ by rotation around the point $\tp(\An)=\tm(\An)$ through $2\pi/3$.
(Rotating once more, we obtain a nice tessellation of the Euclidean equilateral triangle.)
\end{example}

\begin{figure}[h]
$$
\begin{array}{c}
\includegraphics[width=7cm]{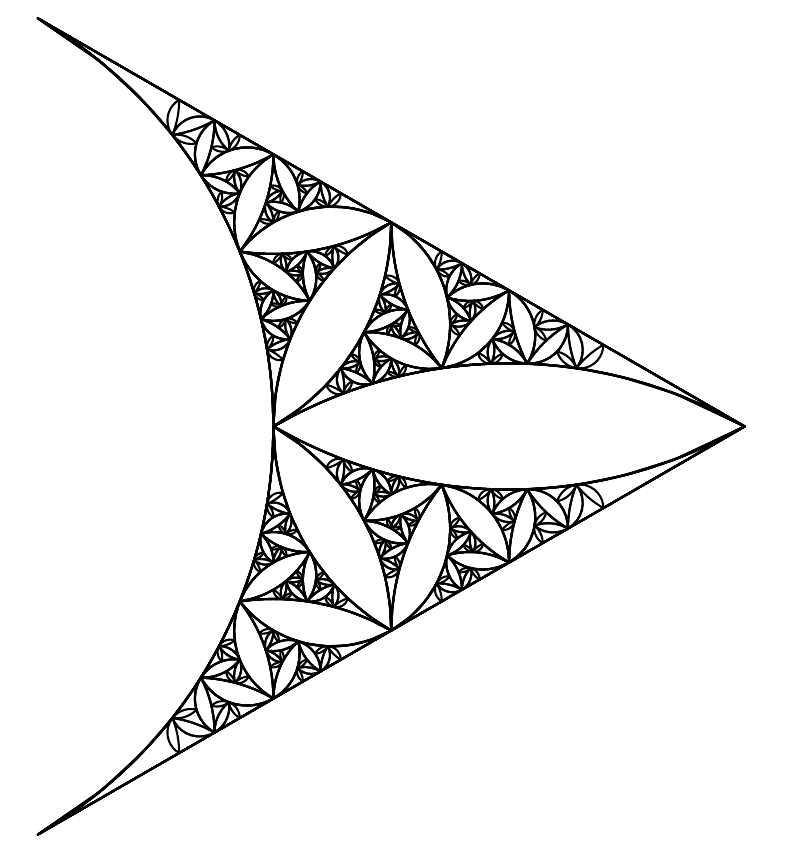}
\end{array}
\begin{array}{c}
\includegraphics[width=8cm]{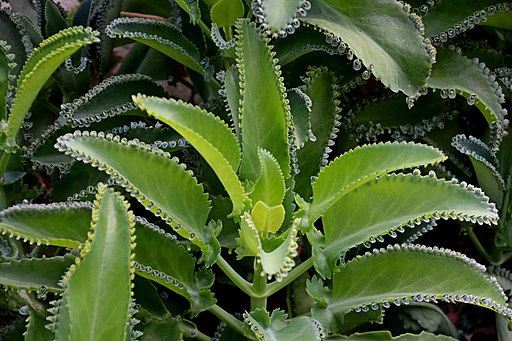}
\end{array}
$$
\caption{Canonical bryophyllum $\Br_{2\pi/3,\pi/6}$ (left), Bryophyllum laetivirens~\cite{Gol} (right).}
\label{fig-classic-bryo}
\end{figure}

\begin{remark}
The name of the fractal object that we study in this paper, bryophyllum, was inspired by the similarity between the leaves of the plant Kalanchoe (bryophyllum laetivirens) (on the right in Figure~\ref{fig-classic-bryo}) and the image of the canonical bryophyllum $\Br_{2\pi/3,\pi/6}$ (on the left in Figure~\ref{fig-classic-bryo}).
\end{remark}


\begin{example}
\label{exam1}
Now let us consider the case of the canonical bryophyllum $\Br_{\pi/2,\pi/5}$,
see Figure~\ref{4mapping-123}.
We show the domain $\An\Ap\Am$ in the middle of Figure~\ref{4mapping-123},
while its images under the mappings $\tm$ and $\tp$ are on the left and on the right respectively.
Figure~\ref{4mapping-456} shows all iterates of depth one, two, and three for $\Br_{\pi/2,\pi/5}$.

\begin{figure}[t]
$$
\begin{array}{c}
\includegraphics{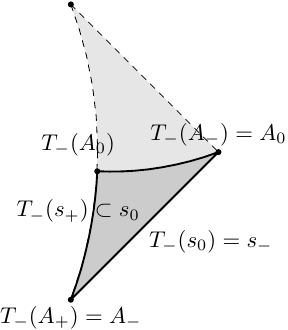}
\end{array}
\longleftarrow
\begin{array}{c}
\includegraphics{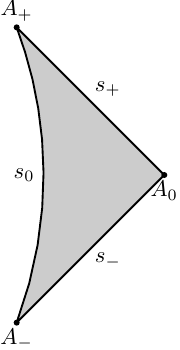}
\end{array}
\longrightarrow
\begin{array}{c}
\includegraphics{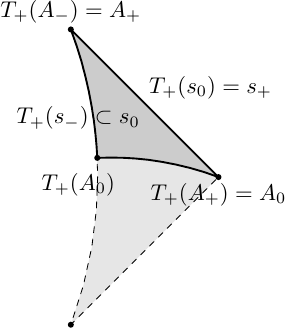}
\end{array}
$$
\caption{The triangle $\An\Ap\Am$ for the canonical bryophyllum $\Br_{\pi/2,\pi/5}$ (middle),
and its images under the mappings~$\tm$ (left) and~$\tp$ (right).}
\label{4mapping-123}
\end{figure}

\begin{figure}[t]
$$
\begin{array}{c}
\includegraphics{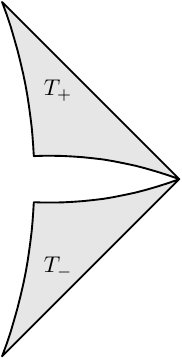}
\end{array}
\longrightarrow
\begin{array}{c}
\includegraphics{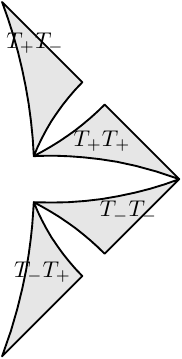}
\end{array}
\longrightarrow
\begin{array}{c}
\includegraphics{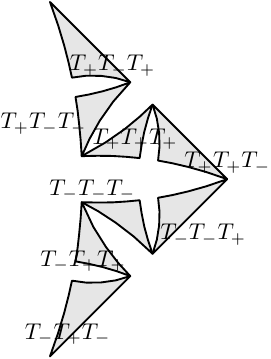}
\end{array}
$$
\caption{All iterates of depth one, two, and three for $\Br_{\pi/2,\pi/5}$.}
\label{4mapping-456}
\end{figure} 
\end{example}

\begin{example}
\label{exam3}
\label{Peano-example}
Let us finally consider the case of the canonical bryophyllum $\Br_{\pi/2,\pi/4}$, see Figure~\ref{fig-peano-bryo}.
\begin{figure}
\centering
\includegraphics[width=2cm]{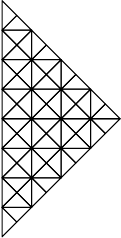}
\caption{Canonical Peano bryophyllum $\Br_{\pi/2,\pi/4}$.}
\label{fig-peano-bryo}
\end{figure}
Here all possible compositions of~$\tp$ and~$\tm$
of length~$n$ represent the construction of the Peano space-filling curve on the Euclidean right-angled isosceles triangle.
We call this bryophyllum the {\it Peano bryophyllum}.
The maps~$\tpm$ for the Peano bryophyllum are affine.
The set of all affine bryophylla is discussed below in Subsection~\ref{subsection-affine}.
\end{example}

\subsection{Affine Canonical Bryophylla}
\label{subsection-affine}

Let us start with the following general definition:

\begin{definition}
We call a canonical bryophyllum {\it affine\/}
if the maps~$\tp$ and~$\tm$ are represented by affine transformations.
\end{definition}

\bigskip\noindent
We have the following description of affine canonical bryophylla:

\begin{proposition}
\label{prop-affine}
A canonical bryophyllum $\Br_{\varphi,\psi}$ is affine if and only if
\begin{equation*}
\label{equation-affine}
\psi=\varphi/2.
\end{equation*}
\end{proposition}

\bigskip\noindent
We will use the following lemma that we will prove later in Subsection~\ref{subsection-centre-sn}:

\begin{lemma}
\label{lemma-denominator}
Let $0<\varphi <\pi$, then 
\[\cos(\varphi)-r\cos(\psi)=\frac{\sin(\varphi+\psi)\sin(\psi-\varphi/2)}{\cos(\varphi/2)}.\]
\end{lemma}

\bigskip\noindent
{\bf Proof of Proposition~\ref{prop-affine}:}
Suppose that the canonical bryophyllum $\Br_{\varphi,\psi}$
is affine.
The map~$\tp$ maps the edge~$\sn$ to the straight line segment~$\sp$.
If the map~$\tp$ is affine, this implies that the edge~$\sn$ is also a straight line segment.
According to Definition~\ref{def-canon-bryo}, this implies $\psi=\varphi/2$.

On the other hand, suppose that $\psi=\varphi/2$,
then the maps~$\tpm$ are of the form
\[
  \tpm(z)=\frac{\al z+\be}{\ga z+\de},
  \quad\text{where}\quad
  \ga=r\cos(\varphi/2)\cdot e^{\pm i(\psi-\varphi/2)}-\cos(\varphi).
\]
Substituting~$\psi=\varphi/2$ and using Lemma~\ref{lemma-denominator}, we obtain
\[
  \ga
  =r\cos(\varphi/2)-\cos(\varphi)
  =-\frac{\sin(\varphi+\psi)\sin(\psi-\varphi/2)}{\cos(\varphi/2)}
  =0,
\]
hence the maps~$\tpm$ are affine.\qed

\begin{example}
\label{example-trio-symm}
Let us consider the case of the canonical bryophyllum $\Br_{2\pi/3,\pi/3}$.
This case is excluded in Definition~\ref{def-canon-bryo} as $\varphi+\psi=\pi$ but it satisfies the conditions of Proposition~\ref{prop-affine}, so would fit into the family of affine canonical bryophylla.
The vertices $\An=1$, $\Ap=e^{2\pi i/3}$, $\Am=e^{-2\pi i/3}$ are the vertices of an Euclidean equilateral triangle with centre at the origin.
The edges $\sp$, $\sn$, $\sm$ are the straight line segments between the vertices.
The maps~$\tpm$ are the symmetries of the triangle given by the rotations through~$\pm2\pi/3$ around the origin.
\end{example}

\subsection{Farey Canonical Bryophylla}
\label{subsection-farey}

All examples in Subsection~\ref{subsection-examples} have 
the following property that we will call the {\bf Edge Condition}:
The edges $\sp$, $\sm$, $\sn$ intersect only in the vertices $\Ap$, $\Am$, $\An$ respectively and bound an open connected curvilinear triangle.
We will investigate which canonical bryophylla have this property.

The classical Farey tessellation has the following property that we will call the {\bf Inclusion Condition}: 
The images of the ideal triangle~$\mathcal T$ under the maps~$\tpm$ are contained in~$\mathcal T$.

We will extend these conditions to bryophylla.


\begin{definition}
\label{def-farey-bryo}
Consider a conformal bryophyllum 
$B=(\An,\Ap,\Am,\sp,\sm,\sn,\tp,\tm)$ whose edges satisfy the
\begin{itemize}
\item{\bf Edge Condition:}
The arcs $\sp$, $\sm$, and $\sn$ intersect only in the vertices.
\end{itemize}
In this case we set $\Delta =\Delta_{B}$ to be the connected component of the complement of
$\sp\cup\sm \cup\sn$ which does not contain the point $(\sp\cap\sm)\setminus\{\An\}$.
Then $B$ is said to be 
{\it Farey} if the following two conditions are satisfied:

\begin{itemize}
\item
{\bf Inclusion Condition:} 
$
\tp(\trio)\subset\trio
\quad\hbox{and}\quad
\tm(\trio)\subset\trio.
$
\item
{\bf Separation Condition:} The images $T_{\pm}(\trio)$ do not overlap:
$\tp(\trio)\cap\tm(\trio)=\{A_0\}$.
\end{itemize}
\end{definition}

We observe the following important property of Farey bryophilla:

\begin{proposition}
All inner angles of~$\trioB$ for a Farey bryophyllum~$B$ do not exceed $\pi$.
\end{proposition}

\begin{proof}
Consider the composition of mappings~$\tp\tm$.
We have
$$
\tp\tm(\Ap)=\Ap;
\quad 
\tp\tm(\An)\in\sp;
\quad 
\tp\tm(\Am)\in\sn.
$$
For Farey bryophylla,
the images of~$\An$ and~$\Am$ are interior points of the arcs. So if the angle at~$\An$ or at~$\Am$ is greater than~$\pi$,
then we have that $\tp\tm(\trio)\not\subset\trio$,
and hence either $\tp(\trio)\not\subset\trio$ or $\tm(\trio)\not\subset\trio$.
Therefore, the angles at~$\An$ and at~$\Am$ cannot exceed~$\pi$.
Since the angles at~$\Am$ and~$\Ap$ coincide,
the angle at $\Ap$ cannot exceed~$\pi$ either.
We have shown that the statement holds for all angles of the bryophyllum.
\end{proof}

\begin{remark}
Note that the property of a conformal bryophyllum to be Farey is preserved under conformal equivalence. 
\end{remark}

The Farey canonical bryophylla are characterised as follows:

\begin{theorem}
\label{theorem-Farey-region}

A canonical bryophyllum $\Br_{\varphi,\psi}$ is Farey if the pair~$(\varphi, \psi)$ satisfies the following system of inequalities:
$$
\left\{
\begin{array}{l}
\sin{\psi}+2\sin\Big(\psi-\frac{\varphi}{2}\Big)\cos\frac{\varphi}{2}\ge0;\\
\sin\psi\cos\varphi-2\sin\Big(\psi-\frac{\varphi}{2}\Big)\cos\frac{\varphi}{2}\ge 0;\\
\varphi + 2\psi <\pi.\\
\end{array}
\right.
$$
\end{theorem}

\noindent
We will prove this theorem later in Subsection~\ref{theorem-Farey-region-proof}.


\subsection{Configuration Space of Canonical Bryophylla}
\label{subsection-config}

The square in Figure~\ref{pic1} with coordinates $(\varphi,\psi)$ represents the configuration space of canonical bryophylla $\Br_{\varphi,\psi}$.
We will describe some important components of the configuration space.

\begin{figure}[t]
\centering
\includegraphics[width=6cm]{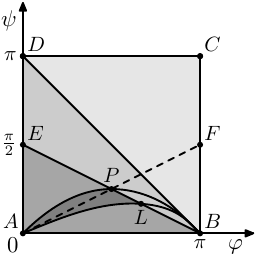}
\caption{Configuration Space of Canonical Bryophylla.}
\label{pic1}
\end{figure}

\bigskip
\noindent
{\it Interesting points:}
\begin{itemize}
\item Point $P$ with coordinates $(\varphi,\psi)=(\pi/2,\pi/4)$ corresponds to Example~\ref{Peano-example} (Peano bryophyllum).
\item Point $L$ with coordinates $(\varphi,\psi)=(2\pi/3, \pi/6)$ corresponds to Example~\ref{example-L}.
\end{itemize}

\bigskip
\noindent
{\it Curves and segments:}
\begin{itemize}
\item
The segment $BE$ corresponds to $r=0$.
\item
The segment $BD$ is excluded in Definition~\ref{def-canon-bryo} and would correspond to $r=-1$.
\item
The segments $AB$ and $CD$ correspond to $r=1$.
These are the configurations for which $\An$ is a fixed point for both mappings $\tm$ and $\tp$,
and so the corresponding maps are not bijective.
We do not consider such cases.
\item
The segments $AD$ and $BC$ correspond to $\varphi=0$ and $\varphi=\pi$ respectively.
Here $\Am=\Ap$, and hence we do not consider such cases.
\item
The curve through $A$, $L$ and $B$ is defined by the first inequality in Theorem~\ref{theorem-Farey-region} and corresponds to the case when the arcs~$\sp$ and~$\sm$ are tangent to~$\sn$.

\item
The curve through~$A$, $P$ and~$B$ is defined by the second inequality in Theorem~\ref{theorem-Farey-region} and corresponds to the case when the arcs~$\tp(\sp)$ and~$\tm(\sm)$ are tangent to each other at the point~$\An$.

\item
Segment~$AF$ is defined by~$\psi=\varphi/2$
and corresponds to affine bryophylla
(see Proposition~\ref{equation-affine}).
The point with coordinates $(\varphi,\psi)=(2\pi/3,\pi/3)$ (Example~\ref{example-trio-symm}) lies on the intersection of the segments~$BD$ and~$AF$.
\end{itemize}

\bigskip
\noindent
{\it Various regions:}
\begin{itemize}
\item
The curvilinear triangle~$APL$ consists of Farey canonical bryophylla described in Theorem~\ref{theorem-Farey-region},
with the first, second, and third inequality corresponding to the sides $AL$, $AP$, and $PL$ respectively.
\item
The triangular region~$BCD$ corresponds to $r<-1$.
We do not consider such cases as the mappings $\tpm$ do not send the curvilinear triangle~$\An\Ap\Am$ into itself. 
\item
The region bounded by the segment~$AB$ and the curve~$ALB$ corresponds to the case when the sides of the curvilinear triangle~$\An\Ap\Am$ do not bound a domain.
\item
The region bounded by the segment~$AB$ and the curve~$APB$ corresponds to the case when $\tm(\sp)\cup\tp(\sm)=\{\An\}$.
This region includes the curvilinear triangle~$APL$.
\item
The curvilinear region~$APE$ contains canonical bryophilla with self-intersections.
\item
The triangle~$BED$ corresponds to the case where the images~$\tp(\sm)$ and~$\tm(\sp)$ have an arc in common.
This region is not of interest to us either.
\end{itemize}

\section{Structural Properties of Bryophylla}
\label{Section-on-convergency}

In this section we discuss some dynamical features of Farey bryophylla. 
We will focus on the case of canonical Farey bryophylla but the statements are true for any conformal bryophyllum as it is conformally equivalent to a canonical one.

Let $B=(\An,\Ap,\Am,\sp,\sm,\sn,\tp,\tm)$ be a canonical Farey bryophyllum.
Let $\trioB$ be the triangle bounded by circlines~$\sp$, $\sm$, $\sn$.

We start in Subsection~\ref{section-fractal-properties}
with convergence properties of nested images of~$\trioB$. 
Later in Subsection~\ref{limit-param} we consider the
limiting invariant sets and describe their natural parametrisations.
Finally in Subsections~\ref{Subsection-Farey-Coordinate} and~\ref{Subsection-Gauss-Map}
we say a few words about a natural ergodic measure defined on Farey bryophylla.

\subsection{Convergence Property of Bryophylla}
\label{section-fractal-properties}

In this subsection we study the images of~$\trioB$
and their convergence.
We will start by introducing a notation that allows us to enumerate all $n$-th iterates of~$\trioB$
by binary sequences of length~$n$
in such a way that the sequences
from~$(1,\dots,1)$ to~$(0,\dots,0)$
in the lexicographic order
correspond to the iterates of~$\trioB$
from top to bottom, i.e.\ from $\Ap$ to $\Am$. 

\begin{definition}
\label{def-nested-trio}
For any binary sequence $x$ we define 
a curvilinear triangle $\trioB(x)$ using the following iterative rules:

\begin{itemize}
\item 
$\De_{B}()=\trioB$;
\item
For $x=(x_1,\dots,x_n)$ with $n\ge 1$
we set
$$
\De_{B}(x_1,\dots,x_{n})=
\left\{
\begin{array}{l}
(\tp\circ\conj)\big(\trioB(x_2,\dots,x_{n})\big)
\quad\hbox{if}~x_{1}=1;\\
(\tm\circ\conj)\big(\trioB(x_2,\dots,x_{n})\big)
\quad\hbox{if}~x_{1}=0,\\
\end{array}
\right.
$$
where $\conj$ is the complex conjugation. \\
(For a general conformal bryophyllum, $\conj$ is the conformal reflection with respect to the line of symmetry of the bryophillum. For a canonical bryophyllum, the line of symmetry coincides with the real line.) 
\end{itemize}
\end{definition}

Let us state the following convergence result:
\begin{theorem}
\label{convergence-triangles}
For any infinite binary sequence $x=(x_k)_{k=1}^\infty$,
the sequence of triangles $\trioB(x_1,\dots,x_n)$
as defined in Definition~\ref{def-nested-trio}
is nested, i.e.
\[\trioB(x_1,\dots,x_n)\subset\trioB(x_1,\dots,x_n,x_{n+1}),\]
and their intersection 
\[\bigcap_{n=1}^{\infty}\trioB(x_1,\dots,x_n).\]
is a one-point set.
We denote the only point in this set by ${\dot\trio}_B(x)$.
\end{theorem}

\noindent
We will prove this theorem later in Subsection~\ref{convergence-triangles-proof}.


\subsection{Limiting Invariant Sets and Their Parametrisations}
\label{limit-param}

\begin{definition}
The {\it limiting invariant set}~$\calD_B$ of~$B$ is
defined as
\[
\calD_B=\bigcap_{k=0}^\infty
\left(\bigcup_{x\in\{0,1\}^k}\trioB(x)\right).
\]  
\end{definition}

It is interesting to note that the decimal representations of rational numbers together with Theorem~\ref{convergence-triangles} deliver a natural mapping of the segment $[0,1]$ to the invariant limiting set $\calD_B$. In order to get this map it is sufficient to prove the following statement.

\begin{proposition}
\label{eQal}
\vspace{2mm}
Denote by $(x_1,\ldots, x_k,\overline{x_{k-1}, \ldots, x_n})$
the infinite eventually periodic sequence 
with a pre-period  $(x_1,\dots,x_k)$
and a period $(x_{k-1}, \ldots, x_n)$.
Then
$$
\trioB(x_1,\dots,x_n,0,\overline{1})=
\trioB(x_1,\dots,x_n,1,\overline{0}).
$$
\end{proposition}

\begin{proof}
By construction, all the triangles
$$
\trioB(x_1,\dots,x_n,0,1,\ldots,1)
\quad\hbox{and}
\quad
\trioB(x_1,\dots,x_n,1,0,\ldots,0)
$$
contain a common vertex $V$
which is in fact one of the vertices of~$\trioB(x_1,\dots,x_n$).
Hence by Theorem~\ref{convergence-triangles} we have
$$
  {\dot\trio}_B(x_1,\dots,x_n,0,\overline{1})
  =V
  ={\dot\trio}_B(x_1,\dots,x_n,1,\overline{0}).
  \qedhere
$$
\end{proof}

\noindent
For $\al\in[0,1]$,
let $\hat\al$ be an infinite sequence of digits of a binary expansion of~$\al$. 

\begin{definition}
\label{defPi}
Let the map $\Pi_B:[0,1]\to \calD_B$ 
be given by $\Pi_B(\al)={\dot\trio}_B(\hat\al)$.
\end{definition}

\begin{remark}
Note that we have two different infinite expansions
for rational numbers of the form~$p/10^k$.
However, Proposition~\ref{eQal} shows that these two expansions lead to the same result in Definition~\ref{defPi}.
\end{remark}

\noindent
Theorem~\ref{convergence-triangles} implies the following result:

\begin{corollary}
The map $\Pi_B$ is continuous.
\qed
\end{corollary}

\subsection{The Farey Coordinate}
\label{Subsection-Farey-Coordinate}



Recall that the {\it Farey sum\/} of two non-negative rational numbers~$\frac{p}{q}$ and~$\frac{r}{s}$ with co-prime non-negative numerators and denominators is defined by
\[\frac{p}{q}\oplus\frac{r}{s}=\frac{p+q}{r+s}.\]
Here we represent~$0$ by~$\frac{0}{1}$.

We will use the following notation:
$\z_{\ge0}=\{n\in\z\st n\ge0\}$.

\begin{definition}
The {\it Farey coordinate\/}~$\eta$ on $[0,1]$ is a function defined iteratively as follows:
\begin{itemize}
\item $\eta(0)=\frac{0}{1},\quad\eta(1)=\frac{1}{1}$.
\item $\eta$ is defined on the points of the form 
$\frac{m}{2^n}$ with~$m,n\in\z_{\ge0}$, $m\ge1$, $m$ odd, as
\[
  \eta\left(\frac{m}{2^{n}}\right)
  =\eta\left(\frac{m-1}{2^n}\right)
  \oplus\eta\left(\frac{m+1}{2^n}\right).
\]  
\end{itemize}
It is clear that $\eta$ is monotone on the set of finite binary rational numbers and can therefore be extended to the whole interval~$[0,1]$ as follows:
\begin{itemize}
\item
For a real number~$\al\in[0,1]$, let 
\[
  \eta(\al)
  =\sup
  \left\{
        \eta\left(\frac{m}{2^n}\right)
        \,\Big|\,
        \frac{m}{2^n}\le\al,~m,n\in\z_{\ge0}
  \right\}.
\]
\end{itemize}
\end{definition}

\begin{remark}
It is clear that the function $\eta$ is continuous on $[0,1]$.
\end{remark}


\noindent
There is a nice formula for the Farey coordinate via so-called LR-decomposition of binary sequences.

\begin{definition}
\label{def-Farey-coord}
Consider $\al\in[0,1)$.
Let $x(\al)$ be the sequence of all digits in an infinite binary expansion of~$\al$ starting with~$0$ for the integer part of $\al$.
If $x(\al)$ starts with $a_1$~zeroes, followed by $a_2$~ones, followed by $a_3$~zeroes and so on then the expression
\[L^{a_1}R^{a_2}L^{a_3}R^{a_4}L^{a_5}\ldots\]
is called an {\it LR-decomposition\/} of~$\al$.
In the case of the sequence~$x(\al)$ having an infinite tail of ones or zeroes, the last~$a_k$ is set to~$\infty$.
\end{definition}

\noindent
We have the following result (equivalent to Theorem~27.25 in~\cite{Karpenkov-book}):

\begin{proposition}
\label{cf-propos}
Consider $\al\in[0,1)$.
Let $L^{a_1}R^{a_2}L^{a_3}\dots$.
be an LR-sequence for $x(\al)$.
Then the Farey coordinate of~$\al$ can be expressed
as a continued fraction:
\[\eta(\al)=[0;a_1:a_2: a_3:\dots].\]
If $x(\al)$ ends with an infinite tail of ones or zeros,
we formally write
\[[0;a_1:a_2:\ldots:a_n:\infty]=[0;a_1:a_2:\ldots:a_n].\qed\]
\end{proposition}

\begin{example}
Consider~$\al=1/4$.
From Definition~\ref{def-Farey-coord} we obtain the following value of the Farey coordinate~$\eta(\al)$ of~$\al$:
\begin{align*}
\eta(\al)=\eta\left(\frac{1}{4}\right)
&=
\eta\left(\frac{1-1}{4}\right)\oplus 
\eta\left(\frac{1+1}{4}\right)=
\eta(0)\oplus\eta\left(\frac{1}{2}\right)\\
&=\eta(0)\oplus
\left(\eta\left(\frac{1-1}{2}\right)\oplus \eta\left(\frac{1+1}{2}\right)
\right)
=\eta(0)\oplus\big(\eta(0)\oplus\eta(1)\big)\\
&=\frac{0+0+1}{1+1+1}=\frac{1}{3}.
\end{align*}
On the other hand, we have two binary expansions for~$\al=1/4$,
\[\al=0.01\bar 0=0.00\bar 1.\]
The corresponding infinite LR-sequences are
$L^2R^1L^\infty$ and $L^3R^{\infty}$.
From these two LR-sequences for~$\al$ we obtain two continued fraction expansions for~$\eta(\al)$
using Proposition~\ref{cf-propos}:
\[
  \eta(\al)=[0;2:1:\infty]=[0;2:1]=\frac{1}{3}
  \quad\text{and}\quad
  \eta(\al)=[0;3:\infty]=[0;3]=\frac{1}{3}.
\]
We note that both continued fraction expansions lead to the same value of~$\eta(\al)$.
\end{example}

\begin{proposition}
The Farey coordinate is a continuous monotone bijective map of the segment $[0,1]$ to itself.
\qed
\end{proposition}

\subsection{Gauss Map and the Induced Ergodic Measure on Farey Bryophylla}
\label{Subsection-Gauss-Map}

In this subsection we discuss a natural ergodic measure on Farey bryophylla induced by the Farey coordinate,
for more details see~\cite{Karpenkov-book}.

Recall that $B$ is a canonical Farey bryophyllum.
Denote by~$\eta_B$ the lifting of the Farey coordinate~$\eta$ 
to~$\calD_B$ via~$\Pi_{B}$.
Now we have a natural Gauss map~$\Ga_B$ defined on~$\Pi_{B}$.
It is defined by forgetting the first coordinate of the continued fraction expression for~$\eta_B$:
\[\Ga_B([0;a_1:a_2:a_3:\cdots])=[0;a_2:a_3:\cdots].\]
This map has a remarkable (fractal) action on curvilinear triangles~$\trioB(x_1,\dots,x_n)$.
We have
\[
  \Ga_B\big(\trioB(x_1,\dots,x_{n})\big)
  =\trioB(x_2,\dots,x_{n}).
\]
We can see that the action of the Gauss map~$\Ga_B$ on the Farey coordinate coincides with the Gauss map on the unit segment.
Therefore, we have an induced ergodic measure on~$\calD_B$, lifted from the continued fraction ergodic measure~$\hat\mu$  defined on a measurable set $S\subset[0,1]$ as
\[
  \hat\mu(S)
  =\frac{1}{\ln(2)}\int\limits_S\,\frac{dx}{1+x}.
\]

\section{Proofs of the Main Results}
\label{Proofs of the main results}

\noindent
In this section we collect the proofs of~Theorems~\ref{theorem-clasification}, \ref{theorem-Farey-region},
and~\ref{convergence-triangles}.

\subsection{Proof of Theorem~\ref{theorem-clasification}(1)}
\label{theorem-clasification-1-proof}

Let $(\An,\Ap,\Am,\sp,\sm,\sn,\tp,\tm)$
be a conformal bryophyllum.
We will first show how to bring the points $\An,\Ap,\Am$ into a standard position using conformal maps.
Let $\ell_{\pm}$ and $\ell_0$ be the circlines that contain the arcs/segments $\spm$ and $\sn$ respectively.
By definition, the circlines~$\ell_+$ and~$\ell_-$ do not coincide and hence they intersect in exactly two points.
One of these two points is~$\An$.
A conformal map that takes the other intersection point of the circlines~$\ell_+$ and~$\ell_-$ to~$\infty$ maps~$\ell_{\pm}$ to circlines going through~$\infty$, i.e.\ to straight lines.
Thus we can assume without loss of generality
that the circlines~$\ell_{\pm}$ are straight lines.
Note that the condition $\sp\cap\sm=\{\An\}$
implies that the arcs~$\spm$ do not pass through~$\infty$ and therefore are the straight line segments $s_{\pm}=[\Apm,\An]$ between $\Apm$ and $\An$.
There exists an affine map that maps the circle through the points $\An$, $\Ap$, $\Am$ to the unit circle and maps~$\An$ to~$1$,
hence we can assume without loss of generality that the points $\An$, $\Ap$, $\Am$ are on the unit circle and that $\An=1$.
Now $\ell_0$ is a circle, while $\ell_{\pm}$ are secants of the circle~$\ell_0$ from the point~$\An$.
Recall that $\angle(\sn,\sp)=\angle(\sn,\sm)$
(Proposition~\ref{angles}),
hence the secants $\ell_{\pm}$ make the same angle with the circle~$\ell_0$.
In other words, the line through~$\An=1$ and the centre of the circle~$\ell_0$ is a symmetry axis of the configuration $\ell_0$, $\ell_{+}$, $\ell_{-}$.
It follows that the points $\Ap$ and $\Am$ on the unit circle are conjugates of each other,
i.e. $\Ap$ and~$\Am$ correspond to complex numbers~$e^{i\varphi}$ and~$e^{-i\varphi}$ for some~$\varphi\in[0,2\pi]$.
Due to the symmetry of the point configuration we can assume that $\varphi\in[0,\pi]$.
The condition that the points~$\Ap=e^{i\varphi}$ and~$\Am=e^{-i\varphi}$ are distinct excludes the cases $\varphi=0$ and~$\varphi=\pi$,
hence $\varphi\in(0,\pi)$.

\bigskip
We will now calculate the maps~$\tpm$.
Let $Q=\tp(\An)$ and let $q=re^{i\psi}$ be the corresponding complex number.
Since the points $\Ap$, $\Am$ and $\An$ are distinct,
the map $\tp$ is uniquely determined by the conditions
$$\tp(\Ap)=\An,\quad\tp(\Am)=\Ap,\quad\tp(\An)=Q.$$
The complex coordinates of the points $\Ap$, $\Am$, $\An$ and $Q$ are
$$a=e^{i\varphi},~\bar{a}=e^{-i\varphi},~1~\text{and}~q=re^{i\psi}$$
respectively,
so the map~$\tp$ satisfies the conditions
$$\tp(a)=1,\quad\tp(\bar{a})=a,\quad\tp(1)=q.$$
Direct computations show that the map given by
$$\tp(z)=\frac{(2q-a-1)z+(1+a-qa-q\bar{a})}{(q+q\bar{a}-a-\bar{a})z+(2-q-q\bar{a})}$$
satisfies these conditions:
\begin{align*}
  \tp(a)&=\frac{(2q-a-1)a+(1+a-qa-q\bar a)}{(q+q\bar a-a-\bar a)a+(2-q-q\bar a)}=\frac{1-a^2+qa-q\bar a}{1-a^2+qa-q\bar a}=1,\\
  \tp(\bar a)&=\frac{(2q-a-1)\bar a+(1+a-qa-q\bar a)}{(q+q\bar a-a-\bar a)\bar a+(2-q-q\bar a)}=\frac{a(1+q\bar a^2-q-\bar a^2)}{1+q\bar a^2-q-\bar a^2}=a,\\
  \tp(1)&=\frac{(2q-a-1)+(1+a-qa-q\bar a)}{(q+q\bar a-a-\bar a)+(2-q-q\bar a)}=\frac{q(2-a-\bar a)}{2-a-\bar a}=q.
\end{align*}
Let us now compute $\tp(Q)$:
\begin{align*}
  \tp(q)
  &=\frac{(2q-a-1)q+(1+a-qa-q\bar a)}{(q+q\bar a-a-\bar a)q+(2-q-q\bar a)}
  =\frac{1-q+2q^2+a-2qa-q\bar a}{2-q+q^2+q^2\bar a-qa-2q\bar a}\\
  &=\frac{\bar a(q-a)(2qa-a-1)}{\bar a(q-a)(qa+q-2)}
  =\frac{2qa-a-1}{qa+q-2}.
\end{align*}
Let us write down the condition for the point $\tp(Q)$ to be on the line $\An\Ap$.
This is equivalent to the fact that the vector $\An\Ap$ is parallel to the vector $\An(\tp(Q))$,
i.e.\ we need
$$\frac{1-\tp(q)}{1-a}$$
to be real.
We calculate
\begin{align*}
  1-\tp(q)
  &=\frac{q+a-qa-1}{qa+q-2}=\frac{(q-1)(1-a)}{qa+q-2},\\
  \frac{1-\tp(q)}{1-a}
  &=\frac{q-1}{qa+q-2}
  =\frac{(q-1)(\bar q\bar a+\bar q-2)}{|qa+q-2|^2}\\
  &=\frac{|q|^2\bar a+|q|^2-2q-\bar q\bar a-\bar q+2}{|qa+q-2|^2}.
\end{align*}
For this number to be real we need $|q|^2\bar a-2q-\bar q\bar a-\bar q$ to be real.
We calculate
$$|q|^2\bar a-2q-\bar q\bar a-\bar q=r^2e^{-i\varphi}-2re^{i\psi}-re^{-i(\psi+\varphi)}-re^{-i\psi}.$$
The imaginary part of this number is
$$r\left(-r\sin(\varphi)+\sin(\psi+\varphi)-\sin(\psi)\right)$$
which vanishes if and only if either $r=0$ or $\sin(\varphi)=\sin(\psi+\varphi)-\sin(\psi)=0$ or
$$\sin(\varphi)\ne0\quad\text{and}\quad r=\frac{\sin(\psi+\varphi)-\sin(\psi)}{\sin(\varphi)}=\frac{\cos(\varphi/2+\psi)}{\cos(\varphi/2)}.$$
Note that the case $\sin(\varphi)=0$ does not correspond to a bryophyllum as in this case the points $\Ap=e^{i\varphi}$ and $\Am=e^{-i\varphi}$ are not distinct.
The case $r=0$ is included in the case $r=\frac{\cos(\varphi/2+\psi)}{\cos(\varphi/2)}$ if $\cos(\varphi/2+\psi)=0$.
Therefore
$$r=\frac{\cos(\varphi/2+\psi)}{\cos(\varphi/2)}\quad\text{and}\quad\sin(\varphi)\ne0.$$

\bigskip\noindent
Note that if we add~$\pi$ to $\psi$, both
$r=\frac{\cos(\varphi/2+\psi)}{\cos(\varphi/2)}$ and $e^{i\psi}$ will change sign,
hence $q=re^{i\psi}$ will remain unchanged.
This means that we will get the same bryophyllum.
For this reason,
it is sufficient to consider values of~$\psi$ in the range~$[0,\pi]$.

\bigskip\noindent
The arc~$\sn$ contains the points $\Ap=e^{i\varphi}$, $\Am=e^{-i\varphi}$ and $Q=re^{i\psi}$.
The condition $\varphi+\psi\ne\pi$ implies that these three points are distinct.
In the case $\psi=\varphi/2$, the points $\Ap$, $\Am$, $Q$ are collinear and the edge~$\sn$ is the vertical straight line segment between~$\Ap$ and~$\Am$.
In the case $\psi\ne\varphi/2$, the points $\Ap$, $\Am$, $Q$ are not collinear and the edge~$\sn$ is an arc of a circle with centre on the real axis.

\bigskip\noindent
We know that
\[
  \tp(z)
  =\frac{(2q-a-1)z+(1+a-qa-q\bar{a})}{(q+q\bar{a}-a-\bar{a})z+(2-q-q\bar{a})}.
\]
We substitute
\[
  a=e^{i\varphi},\quad q=re^{i\psi}
  =\frac{\cos(\varphi/2+\psi)}{\cos(\varphi/2)}\cdot e^{i\psi}
\]
and simplify using
$
 a+\bar{a}
 =e^{i\varphi}+e^{-i\varphi}
 =2\cos(\varphi)
$
and
\begin{align*}
  a+1=e^{i\varphi}+1
  &=\cos(\varphi)+1+i\sin(\varphi)\\
  &=2\cos^2(\varphi/2)+2i\sin(\varphi/2)\cos(\varphi/2)\\
  &=2\cos(\varphi/2)(\cos(\varphi/2)+i\sin(\varphi/2))=2\cos(\varphi/2)\cdot e^{i\varphi/2}
\end{align*}
to see that the map~$\tp$ must be of the form
$$
  \tp(z)
  =\frac
    {(re^{i\psi}-\cos(\varphi/2)\cdot e^{i\varphi/2})z+(\cos(\varphi/2)\cdot e^{i\varphi/2}-r\cos(\varphi)e^{i\psi})}
    {(r\cos(\varphi/2)\cdot e^{i(\psi-\varphi/2)}-\cos(\varphi))z+(1-r\cos(\varphi/2)\cdot e^{i(\psi-\varphi/2)})}.
$$
By symmetry or similar calculation, the map~$\tm$ must be of the form
$$
  \tm(z)
  =\frac
    {(re^{-i\psi}-\cos(\varphi/2)\cdot e^{-i\varphi/2})z+(\cos(\varphi/2)\cdot e^{-i\varphi/2}-r\cos(\varphi)e^{-i\psi})}
    {(r\cos(\varphi/2)\cdot e^{-i(\psi-\varphi/2)}-\cos(\varphi))z+(1-r\cos(\varphi/2)\cdot e^{-i(\psi-\varphi/2)})}.
$$
It remains to show that these maps~$\tpm$ satisfy the conditions in the definition of a conformal bryophyllum.
By construction, $\tp(\Ap)=\An$ and $\tp(\Am)=\Ap$.
We also know that $\tp(\An)=Q\in\sn$ and $\tp(Q)\in\sp=[\Ap,\An]$.
Recall that conformal maps map circlines to circlines.
The map~$\tp$ sends points $\Am,Q,\Ap$ on the arc~$\sn$ to the points $\Ap,\tp(Q),\An$ on the segment~$\sp=[\Ap,\An]$, i.e.\ $\tp(\sn)=\sp$.
The map~$\tp$ also sends points~$\Am,\An$ on the segment~$\sm=[\Am,\An]$ to the points~$\Ap,Q$ on the arc~$\sn$,
hence the image of the segment~$\sm=[\Am,\An]$ under~$\tp$ is the arc through the points~$\Ap$ and~$Q$ that forms an angle~$\angle(\sn,\sm)$ with the segment~$\sp$.
Recall that $\angle(\sn,\sm)=\angle(\sn,\sp)$ (Proposition~\ref{angles}), hence $\sn$ and $\tp(\sm)$ are arcs of the same circle.
Note that the point~$Q=\tp(\An)$ is situated on the arc $\sn$ between the endpoints~$\Am$ and ~$\An$,
hence $\tp(\sm)\subset\sn$ with the same orientation.
The conditions on the map~$\tm$ are fulfilled automatically since the picture is symmetric with respect to the real axis.\qed

\subsection{The Centre of the Arc~$\sn$}
\label{subsection-centre-sn}

Let us find the centre and radius of the arc~$\sn$ of the canonical bryophyllum
$
  \Br_{\varphi,\psi}
  =(\An,\Ap,\Am,\sp,\sm,\sn,\tp,\tm)
$
with $\psi\ne\varphi/2$.
Recall that by Definition~\ref{def-canon-bryo} we have $T_+(A_0)=r e^{ i\psi}$, where 
\[
  r=\frac{\cos(\varphi/2+\psi)}{\cos(\varphi/2)}.
\]


\begin{lemma}
\label{lemma-centre-sn-1}
If $\psi\ne\varphi/2$ then the edge~$\sn$ is an arc of a circle with centre~$x$ and radius~$R$ given by
\[
  x=\frac{1-r^2}{2(\cos(\varphi)-r\cos(\psi))},
  \qquad
  R=\sqrt{1-2x\cos(\varphi)+x^2}.
\]
\end{lemma}

\begin{proof}
By definition, if $\psi\ne\varphi/2$, the edge~$\sn$ is an arc of a circle
through the points~$e^{i\varphi},e^{-i\varphi},re^{i\psi}$.
The points $e^{i\varphi}$ and~$e^{-i\varphi}$ are symmetric with respect to the real axis, hence the centre~$x$ of the circle is on the real axis.
The points~$e^{i\varphi}$ and $re^{i\psi}$ being on the circle with centre~$x$ and radius~$R$ implies
\[R^2=|x-e^{i\varphi}|^2=|x-re^{i\psi}|^2.\]
Rearranging this equation,
we obtain the expressions for~$x$ and~$R$ stated above.
\end{proof}

\bigskip\noindent
We will use Lemma~\ref{lemma-denominator} to simplify the denominator of~$x$:


\bigskip\noindent
{\bf Proof of Lemma~\ref{lemma-denominator}:}
Calculation shows
\begin{align*}
  \cos(\varphi)-r\cos(\psi)
  &=\cos(\varphi)-\frac{\cos(\varphi/2+\psi)}{\cos(\varphi/2)}\cdot\cos(\psi)\\
  &=\frac{\cos(\varphi)\cos(\varphi/2)-\cos(\varphi/2+\psi)\cos(\psi)}{\cos(\varphi/2)}\\
  &=\frac{\frac{1}{2}(\cos(3\varphi/2)+\cos(\varphi/2))-\frac{1}{2}(\cos(\varphi/2+2\psi)+\cos(\varphi/2))}{\cos(\varphi/2)}\\
 &=\frac{\frac{1}{2}(\cos(3\varphi/2)-\cos(\varphi/2+2\psi))}{\cos(\varphi/2)}
 =\frac{\sin(\varphi+\psi)\sin(\psi-\varphi/2)}{\cos(\varphi/2)}.\qquad\qquad\qquad\qquad\qed
\end{align*}

\begin{lemma}
\label{lemma-centre-sn-2}
If $\psi\ne\varphi/2$ then the edge~$\sn$ is an arc of a circle with centre
\[
  x
  =\frac{\sin(\psi)}{2\cos(\varphi/2)\sin(\psi-\varphi/2)}.
\]
For fixed~$\varphi$, the coordinate~$x=x(\psi)$ of the centre of~$\sn$ is monotonously decreasing as a function of~$\psi$, going from~$0$ to~$-\infty$ on the interval $[0,\varphi/2)$ and from~$+\infty$ to~$0$ on the interval~$(\varphi/2,\pi]$.
Hence, the function~$x=x(\psi)$ is injective on the interval~$[0,\pi]$.
\end{lemma}

\begin{proof}
According to Lemma~\ref{lemma-centre-sn-1}, we have
\[
  x=\frac{1-r^2}{2(\cos(\varphi)-r\cos(\psi))}
  \quad\text{with}\quad
  r=\frac{\cos(\varphi/2+\psi)}{\cos(\varphi/2)}.
\]
Substituting
\begin{align*}
  1-r^2
  &=1-\frac{\cos^2(\varphi/2+\psi)}{\cos^2(\varphi/2)}
  =\frac{\cos^2(\varphi/2)-\cos^2(\varphi/2+\psi)}{\cos^2(\varphi/2)}\\
  &=\frac{\frac{1}{2}(\cos(\varphi)+1)-\frac{1}{2}(\cos(\varphi+2\psi)+1)}{\cos^2(\varphi/2)}\\
  &=\frac{\frac{1}{2}(\cos(\varphi)-\cos(\varphi+2\psi))}{\cos^2(\varphi/2)}
  =\frac{\sin(\varphi+\psi)\sin(\psi)}{\cos^2(\varphi/2)}
\end{align*}
and the result of Lemma~\ref{lemma-denominator}
\[
  \cos(\varphi)-r\cos(\psi)
  =\frac{\sin(\varphi+\psi)\sin(\psi-\varphi/2)}{\cos(\varphi/2)}
\]
into the formula for~$x$, we obtain
\begin{align*}
  x=\frac{1-r^2}{2(\cos(\varphi)-r\cos(\psi))}
&  =\frac{\sin(\psi)}{2\sin(\psi-\varphi/2)\cos(\varphi/2)}.
\end{align*}
To show that $x=x(\psi)$ is monotonously decreasing for $\psi\in[0,\pi]$, we can rewrite~$x(\psi)$ as
\[
  x(\psi)
  =\frac{1}{2\cos^2(\varphi/2)-2\sin(\varphi/2)\cos(\varphi/2)\cot(\psi)}
\]
and recall that $\cot(\psi)$ is monotonously decreasing on~$[0,\pi]$.
Alternatively, we can differentiate $x(\psi)$
to see that
\[
  \frac{dx(\psi)}{d\psi}
  =-\frac{\tan(\varphi/2)}{2\sin^2(\psi-\varphi/2)}<0.
  \qedhere
\]
\end{proof}

\subsection{Proof of Theorem~\ref{theorem-clasification}(2)}
\label{theorem-clasification-2-proof}

Let $\Br_{\varphi,\psi}$ and $\Br_{\varphi',\psi'}$
be two conformally equivalent canonical bryophylla,
i.e.\ there exists a conformal map~$T$ sending vertices and edges of~$\Br_{\varphi,\psi}$
to the corresponding vertices and edges of~$\Br_{\varphi',\psi'}$.
The map~$T$ preserves the angle $\angle(\sp,\sm)=\pi-\varphi$
(Remark~\ref{rem-canon-bryo}), hence $\varphi'=\varphi$.
The vertices and the edges~$\spm$ of the canonical bryophyllum~$\Br_{\varphi,\psi}$ 
are determined by~$\varphi$,
hence they coincide for~$\Br_{\varphi,\psi}$ and~$\Br_{\varphi,\psi'}$.
The map~$T$ preserves the angle $\angle(\sp,\sn)=\angle(\sm,\sn)$,
hence the arcs~$\sn$ coincide for~$\Br_{\varphi,\psi}$ and~$\Br_{\varphi,\psi'}$.
This implies that the centres of the arcs coincide, $x(\psi)=x(\psi')$.
Lemma~\ref{lemma-centre-sn-2} says that, for fixed~$\varphi$, the coordinate~$x=x(\psi)$ of the centre of the arc~$\sn$ is an injective function of~$\psi$, hence $x(\psi)=x(\psi')$ implies $\psi=\psi'$.
\qed

\subsection{Some Circles Associated with Canonical Bryophylla}

Consider the canonical bryophyllum
$
 \Br_{\varphi,\psi}
 =(\An,\Ap,\Am,\sp,\sm,\sn,\tp,\tm)
$
with $\psi\ne\varphi/2$.
Consider the circle through the points~$\Ap$ and~$\Am$ with centre at the point~$t\in\r$.
Let $S_t$ be the arc of this circle situated
to the right of the vertical line through~$\Ap$ and~$\Am$.

\begin{proposition}
\label{prop-S-minus-one}
The circle~$S_t$ with~$t=-1$
is tangent to the line $\Ap\An$ at the point~$\Ap$
and to the line~$\Am\An$ at the point~$\Am$.
\end{proposition}

\begin{proof}
An arc of the circle between~$\Ap$ and~$\Am$ is shown as a dashed arc in Figure~\ref{image-circle-1}.
Consider the Euclidean triangle with vertices $B=-1$, $\Ap=e^{i\varphi}$, $\An=1$.
The angle $\angle B\Ap\An$ is inscribed in the upper semi-circle of the unit circle with centre at the origin,
hence the angle is equal to~$\pi/2$.
This means that the line $\Ap\An$ is perpendicular to the radius~$B\Ap$ of the circle and therefore tangent to the circle at the point~$\Ap$.
Similarly, the line $\Am\An$ is tangent to the circle at the point~$\Am$.
\end{proof}

\begin{proposition}
\label{prop-S-one-over-cos}
Let $O$ be the origin.
Then the circle~$S_t$ with~$t=1/\cos(\varphi)$
is tangent to the line $O\Ap$ at the point~$\Ap$
and to the line~$O\Am$ at the point~$\Am$.
\end{proposition}

\begin{proof}
Let $B$ be the centre of the circle.
The dot-product of the vectors
\[
  B\Ap=\left(\cos(\varphi)-\frac{1}{\cos(\varphi)},\sin(\varphi)\right),\quad
  O\Ap=(\cos(\varphi),\sin(\varphi))
\]
is
\[
  B\Ap\bullet O\Ap
  =\cos^2(\varphi)-1+\sin^2(\varphi)=0,
\]
hence the vectors~$B\Ap$ and~$O\Ap$ are orthogonal,
i.e. the line $O\Ap$ is tangent to the circle~$S_t$ with~$t=1/\cos(\varphi)$.
\end{proof}

\begin{figure}[t]
\begin{center}
\includegraphics{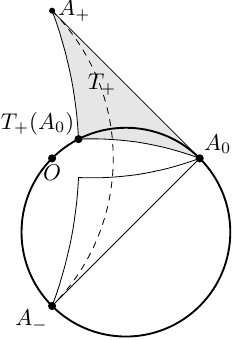}
\end{center}
\caption{The circle passing through the points $O$, $A_-$, $A_0$
and $T_+(A_0)$.}
\label{image-circle-1}
\end{figure} 

\begin{proposition}
\label{place-of-t+}
Let $O$ be the origin.
The circle passing through the points $O$, $\An$, $\Am$
contains the point $\tp(\An)$ and is tangent to the line $\An\Ap$, see Figure~\ref{image-circle-1}.
\end{proposition}

\begin{proof}
Let $R=\frac{1}{2\cos(\varphi/2)}$
and let $C$ be the point
with the complex coordinate $c=R e^{-i\varphi/2}$.
Then we can show that the points $O$, $\An=1$, $\Am=e^{-i\varphi}$ and $\tp(\An)=r e^{i\psi}$
all lie on the circle with centre~$C$ and radius~$R$.
Note that
\[r=\frac{\cos(\varphi/2+\psi)}{\cos(\varphi/2)}=2R\cos(\varphi/2+\psi)\]
and
\[
  c
  =R e^{-i\varphi/2}
  =\frac{\cos(\varphi/2)-i\sin(\varphi/2)}{2\cos(\varphi/2)}
  =\frac{1}{2}-\frac{i}{2}\tan(\varphi/2),
\]
hence
\[c+\bar c=1.\]
We have
\begin{align*}
  |OC|
  &=|Re^{-i\varphi/2}|=R,\\
  |\An C|
  &=|1-c|=|\bar c|=R,\\
  |\Am C|
  &=|e^{-i\varphi}-R e^{-i\varphi/2}|
  =|e^{-i\varphi}(1-R e^{i\varphi/2})|\\
  &=|1-R e^{i\varphi/2}|=|1-\bar c|=|c|=R,\\
  |\tp(\An)C|
  &=|r e^{i\psi}-R e^{-i\varphi/2}|
  =|2R\cos(\varphi/2+\psi)e^{i\psi}-R e^{-i\varphi/2}|\\
  &=R\cdot|(2\cos(\varphi/2+\psi)e^{i(\varphi/2+\psi)}-1)e^{-i\varphi/2}|\\
  &=R\cdot|2\cos(\varphi/2+\psi)e^{i(\varphi/2+\psi)}-1|
  =R\cdot|e^{i(\varphi+2\psi)}|
  =R.
\end{align*}
Now we will look at the angle between the radius $C\An$ of the circle and the line $\An\Ap$.
We have vectors
\begin{align*}
  &C\An
  =1-c=\bar c=R e^{i\varphi/2},\\
  &\An\Ap
  =e^{i\varphi}-1=i\cdot 2\sin(\varphi/2) e^{i\varphi/2}.
\end{align*}
The ratio
\[
  \frac{e^{i\varphi}-1}{1-c}
  =\frac{i\cdot 2\sin(\varphi/2) e^{i\varphi/2}}{R e^{i\varphi/2}}
  =i\cdot\frac{2\sin(\varphi/2)}{R}
\]
is a positive real multiple of~$i$.
This means that the vectors~$\An\Ap$ and~$C\An$ are orthogonal,
i.e. the line $\An\Ap$ is tangent to the circle through the points $O$, $\An$, $\Am$, $\tp(\An)$.
\end{proof}

\subsection{Proof of Theorem~\ref{theorem-Farey-region}}
\label{theorem-Farey-region-proof}

Let
$
 \Br_{\varphi,\psi}
 =(\An,\Ap,\Am,\sp,\sm,\sn,\tp,\tm)
$
be a canonical bryophyllum.
Recall that $S_t$ is the arc of the circle through the points~$\Ap$ and~$\Am$ with centre at~$t\in\r$, to the right of the vertical line through~$\Ap$ and~$\Am$.

\subsubsection{Edge Condition}

\begin{proposition}
\label{proposition-Farey-interior}
The arcs $\sn,\sm,\sp$ of the canonical bryophillum~$\Br_{\varphi,\psi}$ 
bound a curvilinear triangle with angles not exceeding $\pi$ if and only if 
\[\psi\ge\psi_1(\varphi),\]
where $\psi_1(\varphi)$ is the unique solution of the equation
\[\sin(\psi)+2\sin(\psi-\varphi/2)\cos(\varphi/2)=0.\]

\end{proposition}

\begin{proof}
First of all note that the edges~$\sp$ and $\sm$ only intersect in the vertex~$\An$.
For symmetry reasons, the edges~$\sp$ and $\sn$ intersect if and only if the edges~$\sm$ and~$\sn$ intersect.
So it is sufficient to check when the edges~$\sp$ and~$\sn$ intersect.

Proposition~\ref{prop-S-minus-one} says that the angles between $S_{-1}$ (dashed arc in Figure~\ref{image-circle-1}) and the lines $\spm$ are equal to zero.
Let~$x=x(\psi)$
be the centre of the circle determined by the arc~$\sn$
as in Lemma~\ref{lemma-centre-sn-2}.
We conclude that $\sn\cap\spm=\{\Apm\}$ if and only if $\sn$ is situated to the left of the arc~$S_{-1}$
which is equivalent to
\[x\le-1\quad\text{or}\quad x\ge0.\]
The analysis of the function~$x=x(\psi)$
in Lemma~\ref{lemma-centre-sn-2}
shows that there exists a unique value~$\psi_1=\psi_1(\varphi)$
such that $x(\psi_1)=-1$.
Moreover, $x\le-1$ if and only if $\psi_1\le\psi<\varphi/2$
and $x\ge0$ if and only if $\varphi/2<\psi$.
In the case $\psi=\varphi/2$ (affine bryophylla), the edge~$\sn$ is the straight line segment~$\Ap\Am$ and only intersects~$\spm$ at the vertices~$\Apm$.
Hence the arcs $\sn,\sm,\sp$ bound a curvilinear triangle
if and only if~$\psi\ge\psi_1$.
Note that the value~$\psi_1$ satisfies the equation
\[x(\psi_1)=\frac{\sin(\psi_1)}{2\sin(\psi_1-\varphi/2)\cos(\varphi/2)}=-1,\]
hence $\psi=\psi_1$ is a solution of the equation
\[\sin(\psi)+2\sin(\psi-\varphi/2)\cos(\varphi/2)=0.\]
Note that the angles between the arcs~$\sp,\sm,\sn$ do not exceed~$\pi$ by construction.
\end{proof}

\subsubsection{When do $\tp(\trio)$ and $\tm(\trio)$ have only one vertex in common?}


\begin{proposition}
\label{proposition-Farey-separation}
Assume that $\Br_{\varphi,\psi}$ satisfies the conditions of Proposition~\ref{proposition-Farey-interior}, i.e.\ the edges~$\sn,\sm,\sp$ bound a curvilinear triangle~$\trio$.
Then the condition
\[\tp(\trio)\,\cap\,\tm(\trio)=\{\An\}\]
is equivalent to
\[
  \psi<\frac{\pi}{2}-\frac{\varphi}{2}
  \quad\text{and}\quad
  \psi\le\psi_2(\varphi),
\]
where $\psi_2(\varphi)$ is the unique solution of the equation
\[\sin(\psi)\cos(\varphi)-2\sin(\psi-\varphi/2)\cos(\varphi/2)=0.\]
Note that the first inequality is equivalent to $\Im(\tp(\An))>0$,
while the second inequality is equivalent to the arc~$\sn$ being to the right of the arc~$S_t$ with~$t=1/\cos(\varphi)$.
\end{proposition}

\begin{proof}
Let $O$ be the origin.
Note that if $\tp(\sm)\cup\tp(\sp)$ intersects the line~$O\An$
at some point distinct from~$\An$,
then $\tm(\sp)\cup\tm(\sm)$ must intersect the line~$O\An$ at the same point for symmetry reasons, 
hence we have found a point in $\tp(\trio)\,\cap\,\tm(\trio)$
which is distinct from~$\An$.

The borderline configuration where~$\tp(\sm)$ just touches~$O\An$ is when~$\Im(\tp(\An))=0$.
The arc~$\tp(\sm)$ does not intersects~$O\An$ if and only if~$\Im(\tp(\An))>0$.

Now let us check the conditions for $\Im(\tp(\An)>0$.
We have
\[
  \Im(\tp(\An))
  =\Im(r e^{i\psi})
  =r\sin(\psi).
\]
We know that $\sin(\psi)>0$.
According to Remark~\ref{rem-canon-bryo},
we have $r>0$ and hence $\Im(\tp(\An)>0$
if and only if~$\varphi/2+\psi<\pi/2$.

We now need to determine when the arc~$\tp(\sp)$ intersects~$O\An$ at a point other than~$\An$.
By symmetry, this is the case if and only if the arc~$\sn$ intersects the line~$O\Ap$ at a point other than~$\Ap$.  
The borderline configuration is where~$\sn$ is tangent to $O\Ap$.
Proposition~\ref{prop-S-one-over-cos} shows that the arc~$S_t$ with $t=1/\cos(\varphi)$ is tangent to $O\Ap$.
Let~$x=x(\psi)$ be the centre of the circle determined by the arc~$\sn$ as in Lemma~\ref{lemma-centre-sn-2}.
We conclude that $\sn\cap O\Ap=\{\Ap\}$ if and only if $\sn$ is situated to the right of the arc~$S_t$
which is equivalent to
\[x\le0\quad\text{or}\quad x\ge\frac{1}{\cos(\varphi)}.\]
The analysis of the function~$x=x(\psi)$
in Lemma~\ref{lemma-centre-sn-2}
shows that there exists a unique value~$\psi_2=\psi_2(\varphi)$
such that $x(\psi_2)=\frac{1}{\cos(\varphi)}$.
Moreover, $x\ge\frac{1}{\cos(\varphi)}$ if and only if $\varphi/2<\psi\le\psi_2$
and $x\le0$ if and only if $\psi<\varphi/2$.
In the case $\psi=\varphi/2$ (affine bryophylla), the edge~$\sn$ is the straight line segment~$\Ap\Am$ and hence only intersects~$O\Ap$ at the point~$\Ap$.
Hence $\sn\cap O\Ap=\{\Ap\}$
and therefore $\tp(\sp)\cap\r=\{\An\}$
if and only if $\psi\le\psi_2$.
Note that the value~$\psi_2$ satisfies the equation
\[x(\psi_2)=\frac{\sin(\psi_2)}{2\sin(\psi_2-\varphi/2)\cos(\varphi/2)}=\frac{1}{\cos(\varphi)},\]
hence $\psi=\psi_2$ is a solution of the equation
\[
  \sin(\psi)\cos(\varphi)-2\sin(\psi-\varphi/2)\cos(\varphi/2)=0.
  \qedhere
\]


\end{proof}

\subsubsection{When are $\tp(\trio)$ and~$\tm(\trio)$ contained in $\trio$?}


\begin{proposition}
\label{proposition-Farey-inclusion}
Assume that $\Br_{\varphi,\psi}$ satisfies the conditions of Proposition~\ref{proposition-Farey-separation}.
Then the conditions $\tp(\trio)\subset\trio$ and $\tm(\trio)\subset\trio$ are satisfied.
\end{proposition}

\begin{proof}
Note that $\tp(\sn)=\sp$.
The condition $\psi<\frac{\pi}{2}-\frac{\varphi}{2}$ implies
$\Im(\tp(\An))>0$ and therefore $\tp(\sm)\subset\sn$.
It remains to check that $\tp(\sp)\subset\trio$.
Note that $\tp(\sp)$ is an arc. 
First of all, since the angles of $\trio$ do not exceed~$\pi$, the arc $\tp(\sp)$ does not intersect $\sn\setminus\tp(\sm)$.

Now let us consider the angles
$\al=\angle(\sp,\sm)$ and $\beta=\angle(\sn,\sp)=\angle(\sn,\sm)$ 
in the triangle~$\trio$.
For $\psi=\psi_2(\varphi)$ we have that
the arc~$\tp(\sp)$ is tangent to $O\An$.
Therefore the angle between~$\tp(\sp)$ and $\tp(\sn)=\sp$ is equal to the angle~$\angle\Ap\An O=\al/2$.
The map~$\tp$ is conformal, hence $\angle(\tp(\sp),\tp(\sn))=\angle(\sp,\sn)=\beta$.
Thus we have $\beta=\al/2$ for~$\psi=\psi_2(\varphi)$.
For $\psi\le\psi_2(\varphi)$ we have $\beta\le\al/2$.
Now we note that when $\beta\le\al$, the arc $\tp(\sp)$ does not intersect~$\sm$, and hence it is entirely in $\trio$.
\end{proof}

\subsubsection{Conclusion of the Proof:}
Theorem~\ref{theorem-Farey-region} follows directly from Propositions~\ref{proposition-Farey-interior},
\ref{proposition-Farey-separation}, and~\ref{proposition-Farey-inclusion}.
Namely, we have the following:
\begin{itemize}
\item 
{\bf Edge Condition:}
By Proposition~\ref{proposition-Farey-interior},
$\trio$ is a curvilinear triangle with non-intersecting edges
if and only if
\[\psi\ge\psi_1,\]
where $\psi_1=\psi_1(\varphi)$ is the unique solution of the equation
\[\sin(\psi)+2\sin(\psi-\varphi/2)\cos(\varphi/2)=0.\]
\item 
{\bf Separation Condition:}
By Proposition~\ref{proposition-Farey-separation},
the images~$\tp(\trio)$ and~$\tm(\trio)$ only share one vertex
if and only if
\[
  \psi<\frac{\pi}{2}-\frac{\varphi}{2}
  \quad\text{and}\quad
  \psi\le\psi_2,
\]
where $\psi_2=\psi_2(\varphi)$ is the unique solution of the equation
\[\sin(\psi)\cos(\varphi)-2\sin(\psi-\varphi/2)\cos(\varphi/2)=0.\]
\item 
{\bf Inclusion Condition:}
By Proposition~\ref{proposition-Farey-inclusion},
if the conditions of Propositions~\ref{proposition-Farey-interior}
and~\ref{proposition-Farey-separation} are satisfied
then the images~$\tp(\trio)$ and~$\tm(\trio)$ are contained in~$\trio$.
\end{itemize}
These three conditions are precisely the conditions in the definition of a Farey canonical bryophyllum (Definition~\ref{def-farey-bryo}).




\noindent
The conditions $\psi\ge\psi_1(\varphi)$, $\psi\le\psi_2(\varphi)$ and $\psi<\frac{\pi}{2}-\frac{\varphi}{2}$ are equivalent to the first, second and third inequalities in the following system respectively:
$$
\left\{
\begin{array}{l}
\sin{\psi}+2\sin\Big(\psi-\frac{\varphi}{2}\Big)\cos\frac{\varphi}{2}\ge0;\\
\sin\psi\cos\varphi-2\sin\Big(\psi-\frac{\varphi}{2}\Big)\cos\frac{\varphi}{2}\ge 0;\\
\varphi + 2\psi <\pi.\\
\end{array}
\right.
$$

\subsection{Proof of Theorem~\ref{convergence-triangles}}
\label{convergence-triangles-proof}

Let $B=(\An,\Ap,\Am,\sp,\sm,\sn,\tp,\tm)$ be a canonical Farey bryophyllum.
Let $\trioB$ be the triangle bounded by circlines~$\sp$, $\sm$, $\sn$.
Let $x=(x_k)_{k=1}^\infty$ be an infinite binary sequence.
Consider the triangles $\trioB(x_1,\dots,x_n)$ as defined in 
Definition~\ref{def-nested-trio}.
Let $S_B(x_1,\dots,x_n)$ be the circumscribed circle of the triangle~$\trioB(x_1,\dots,x_n)$
and let $D_B(x_1,\dots,x_n)$ be the disc bounded by the circle~$S_B(x_1,\dots,x_n)$.
(For~$n=0$ the notation is $\trioB(x_1,\dots,x_n)=\trioB()=\trioB$, $S_B(x_1,\dots,x_n)=S_B()$, $D_B(x_1,\dots,x_n)=D_B()$.)

\subsubsection{Properties of the Iterates $\trioB(x_1,\dots,x_n)$}
\label{PropertiesIterates}

In this subsection we will investigate the properties of the triangles $\trioB(x_1,\dots,x_n)$ and their circumscribed circles $S_B(x_1,\dots,x_n)$:

\begin{lemma}
\label{lemma-nested-trioB}
Let $i\in\z_{\ge0}$.
\begin{enumerate}[(a)]
\item
\label{part-nested-trioB}
The triangles are nested,
i.e.\ 
$\trioB(x_1,\dots,x_{i+1})\subset\trioB(x_1,\dots,x_i)$.
\item
The triangles~$\trioB(x_1,\dots,x_i)$ and~$\trioB(x_1,\dots,x_{i+1})$ have exactly two vertices in common.
\item
The triangles~$\trioB(x_1,\dots,x_i)$ and~$\trioB(x_1,\dots,x_{i+2})$ have exactly one vertex in common.
\item
We have $\trioB(x_1,\dots,x_i)\subset D_B(x_1,\dots,x_i)$.
\item
\label{part-angles-trioB}
The intersection~$S_B(x_1,\dots,x_i)\cap S_B(x_1,\dots,x_{i+1})$ consists of exactly two points.
The intersection~$D_B(x_1,\dots,x_i)\cap D_B(x_1,\dots,x_{i+1})$ is bounded by two arcs,
and the interior angle between these two arcs does not depend on~$i$.
\item
The circles~$S_B(x_1,\dots,x_i)$ and~$S_B(x_1,\dots,x_{i+2})$ are tangent to each other at the shared vertex 
of the triangles~$\trioB(x_1,\dots,x_i)$ and~$\trioB(x_1,\dots,x_{i+2})$.
Moreover,
$D_B(x_1,\dots,x_{i+2})\subset D_B(x_1,\dots,x_i)$.
\end{enumerate}
\end{lemma}

\begin{proof}
The triangles
$\trioB(x_1,\dots,x_i)$, $\trioB(x_1,\dots,x_{i+1})$, $\trioB(x_1,\dots,x_{i+2})$ are the images under a conformal map of $T_0=\trioB()$, $T_1=\trioB(y_1)$, $T_2=\trioB(y_1,y_2)$,
where $(y_1,y_2)=(x_{i+1},x_{i+2})$.
Therefore the corresponding circles and discs are the images under a conformal map of $S_0=S_B()$, $S_1=S_B(y_1)$, $S_2=S_B(y_1,y_2)$ and $D_0=D_B()$, $D_1=D_B(y_1)$, $D_2=D_B(y_1,y_2)$ respectively.
Inclusion of sets and tangency of circles is preserved under conformal maps, 
hence it is sufficient to show the statements of the lemma for~$i=0$.
\begin{enumerate}[(a)]
\item
By Definition~\ref{def-farey-bryo}, we have~$T_1=(\tpm\circ\conj)(T_0)\subset T_0$.
\item
The triangles~$T_0$ and~$T_1$ have exactly two vertices in common, either $\An,\Ap$ or $\An,\Am$.
\item
The triangles~$T_0$ and~$T_2$ have exactly one vertex in common.
\item
The triangle~$T_0$ is contained in its circumscribed disc~$D_0$.
\item
We have $S_0=S_B()$ and~$S_1=S_B(y)$ with~$y\in\{0,1\}$.
The intersection of the circles~$S_0\cap S_1$ consists of two points. 
The intersection of the discs~$D_0=D_B()$ and~$D_1=D_B(y)$
is bounded by two arcs.
The angle between these arcs does not depend on~$y$
since the circle~$S_B()$ forms the same angle with the circles~$S_B(0)$ and~$S_B(1)$ by symmetry.
\item
By construction, the triangles~$T_0,T_1,T_2$, share a vertex, say~$P$.
Note that $P$ is a vertex of both intersections~$D_0\cap D_1$
and $D_1\cap D_2$. 
By part~(\ref{part-angles-trioB}),
the circles~$S_0$ and~$S_2$ form the same angle with the circle~$S_1$, hence either $S_0$ and~$S_2$ are tangent
or the interiors of the intersections~$D_0\cap D_1$ and~$D_1 \cap D_2$ are on different sides of the circle $S_1$.
The latter case is impossible since 
$T_2\subset T_1\subset T_0$ by part~(\ref{part-nested})
and $T_i\subset D_i$ by part~(\ref{part-T-in-D})
imply $T_2\subset D_0\cap D_1 \cap D_2$.
Therefore the circles~$S_0$ and~$S_2$ must be tangent at the point~$P$,
hence the following configurations are possible:
$D_0\subsetneq D_2$, $D_2\subset D_0$ or $D_0\cap D_2$ is a one-point set.
At least one vertex of~$T_0$ is not shared with~$T_2$ and is on~$S_0$, hence $D_0\subsetneq D_2$ is impossible.
We know that $T_2\subset D_0\cap D_2$ by part~(\ref{part-nested}),
hence $D_0\cap D_2$ cannot consist of only one point.
Therefore the only possible configuration is $D_2\subset D_0$.\qedhere
\end{enumerate}
\end{proof}

\begin{lemma}
\label{lemma-lexicographic}
The binary sequences~$x$ of length~$n$
in the lexicographic order from~$(1,\dots,1)$ to~$(0,\dots,0)$ correspond to the $n$-th iterates of~$\trioB$ from top to bottom, i.e.\ from $\Ap$ to $\Am$.
\end{lemma}

\begin{proof}
First consider~$n=1$. 
The triangles $\trioB(1)=(\tp\circ\conj)(\trioB)$ and $\trioB(0)=(\tm\circ\conj)(\trioB)$ contain the vertices~$\Ap$ and~$\Am$ respectively.

\bigskip
Now assume that the statement holds for $n=k$
and let us prove it for $n=k+1$.
The operator $(\tp\circ\conj)$ sends all triangles
$\trioB(x_1,\dots,x_k)$ to the upper triangle~$\tp(\trioB)$
in the order from $\Ap$ to $\An$.
The operator $(\tm\circ\conj)$ sends all triangles $\trioB(x_1,\dots,x_k)$ to the lower triangle~$\tm(\trioB)$
in the order from $\An$ to $\Am$.
Combining these two sequences of triangles,
we see that triangles $\trioB(x_1,\dots,x_k,x_{k+1})$
in the lexicographic order 
enumerate all $(k+1)$-th iterates from~$\Ap$ to~$\Am$.
\end{proof}

\subsubsection{Properties of the Sequence of Circumscribed Circles}
\label{CircumscribedSequence}

We start with some notation that we will use in the proof:
For~$i\in\z_{\ge0}$,
consider the triangle $T_i=\trioB(x_1,\dots,x_i)$, the circle $S_i=S_B(x_1,\dots,x_i)$, and the disc $D_i=D_B(x_1,\dots,x_i)$.
With this notation, Lemma~\ref{lemma-nested-trioB} becomes:

\begin{corollary}
\label{corollary-nested}
Let~$i\in\z_{\ge0}$.
Then the following statements hold:
\begin{enumerate}[(a)]
\item
\label{part-nested}
The sequence of triangles is nested,
i.e.\ $T_{i+1}\subset T_{i}$.
\item
The triangles~$T_i$ and~$T_{i+1}$ have exactly two vertices in common.
\item
The triangles~$T_i$ and~$T_{i+2}$ have exactly one vertex in common.
\item
\label{part-T-in-D}
We have $T_i\subset D_i$.
\item
\label{part-petals}
The intersection~$S_i\cap S_{i+1}$ consists of exactly two points.
The intersection~$D_i\cap D_{i+1}$ is bounded by two arcs,
and the interior angle between these two arcs
does not depend on~$i$.
\item
\label{part-nested-discs}
The circles~$S_i$ and~$S_{i+2}$ are tangent to each other at the shared vertex of the triangles~$T_i$ and~$T_{i+2}$.
Moreover, $D_{i+2}\subset D_i$.
\end{enumerate}
\end{corollary}

\noindent
Let $O_i$ and $r_i$ be the centre and the radius of the circle~$S_i$

\begin{corollary}
\label{LemmaCenterRadius}
The following two statements hold:
\begin{enumerate}[(i)]
\item
The sequence of points $(O_{2j})$ converges to some point~$\Oeven$
and the sequence of radii $(r_{2j})$ converges to some number $\reven$.
\item
The sequence of points $(O_{2j+1})$ converges to some point~$\Oodd$
and the sequence of radii $(r_{2j+1})$ converges to some number~$\rodd$.
\end{enumerate}
\end{corollary}

\begin{proof}
According to Lemma~\ref{corollary-nested}(\ref{part-nested-discs}),
the sequences of discs $(D_{2j})$ and $(D_{2j+1})$ are nested.
This implies that the sequences $(r_{2j})$ and $(r_{2j+1})$ of their radii are decreasing and non-negative,
therefore they converge to some limits
\[
  \reven=\lim\limits_{j\to\infty}\,r_{2j}
  \quad\text{and}\quad
  \rodd=\lim\limits_{j\to\infty}\,r_{2j+1}.
\]
We can use the limit~$\reven$ to provide a uniform estimate
on the distances between the centers~$O_{2j}$,
hence the sequence~$(O_{2j})$ is Cauchy and therefore converges to some point~$\Oeven$.
A similar argument works for the sequence of points $(O_{2i+1})$.
\end{proof}

\noindent
As a direct consequence of the above we have the following corollary:

\begin{corollary}
\label{LemmaLimitCircles}
The following statements hold:
\begin{enumerate}[(i)]
\item
The intersection~$\Deven$ of all discs $D_{2j}$ is the disc centered at $\Oeven$ of radius~$\reven$.
\item
The intersection~$\Dodd$ of all discs $D_{2j+1}$ is the disc centered at $\Oodd$ of radius~$\rodd$.
\item
Let $\Seven$ and~$\Sodd$ be the boundary circles of the disc~$\Deven$ and~$\Dodd$ respectively,
then the intersection~$\Seven\cap\Sodd$
consists of either one or two points.
\end{enumerate}
\end{corollary}

\begin{proof}
The proofs of statements~$(i)$ and~$(ii)$ are straightforward.
For~$(iii)$, recall that Corollary~\ref{corollary-nested}(\ref{part-petals}) implies that the intersection~$S_i\cap S_{i+1}$
consists of exactly two points
and the angle $\angle(S_i,S_{i+1})$ does not depend on~$i$.
Hence~$\Seven\cap\Sodd\ne\emptyset$ and $\angle(\Seven,\Sodd)=\angle(S_0,S_1)$.
The circles~$S_0$ and~$S_1$ intersect at two distinct vertices
of the triangle~$T_1$.
The third vertex of $T_1$ lies strictly in the interior of the disc~$D_0$, and hence it is not on~$S_0$.
Therefore the angle $\angle(S_0,S_1)=\angle(\Seven,\Sodd)$ is not an integer multiple of~$\pi$.
Thus the circles~$\Seven$ and~$\Sodd$ are not disjoint and do not coincide, so their intersection must consist of one or two points.
\end{proof}

Let us enumerate all vertices of the triangles~$T_i$,
counting every vertex exactly once.
Recall that every two subsequent triangles~$T_i$ and~$T_{i+1}$ have exactly two vertices in common.
Denote by~$P_{-1}$ the vertex of~$T_0$ that is not shared with~$T_1$.
Let $P_0$ and~$P_1$ be the two vertices
shared between ~$T_0$ and~$T_1$.
For every~$i\in\z$, $i\ge 2$, we define $P_i$ inductively to be the vertex of~$T_{i-1}$ not shared with~$T_{i-2}$.


\begin{corollary}
The sequence of vertices~$(P_i)$ is bounded and all its limiting points are contained in~$\Seven\cap\Sodd$.
\end{corollary}

\begin{proof}
Corollary~\ref{LemmaLimitCircles}$(iii)$ implies that $\Seven\cap\Sodd$ consists of one or two points.
Corollaries~\ref{LemmaCenterRadius} and~\ref{LemmaLimitCircles} imply that for every~$\ve>0$ there exists~$N$ such that for any~$i>N$ the circles~$S_{i-1}$ and~$S_i$ are contained
in the $\ve$-neighbourhood of the circles~$\Seven$ and~$\Sodd$,
hence the points of~$S_i\cap S_{i-1}$ accumulate to the points of~$\Seven\cap\Sodd$.
By construction, for every~$i\in\z_{\ge0}$, there exists~$k\in\z_{\ge0}$ such that $P_i\in S_{k-1}\cap S_k$,
and hence the sequence~$(P_i)$ is bounded and all its limiting points are contained in~$\Seven\cap\Sodd$.
\end{proof}

\subsubsection{Finiteness Properties of Cross-Ratios}
Recall that the {\it cross-ratio} of four points $z_1,z_2,z_3,z_4\in \bar{\mathbb C}$ is defined as follows:
$$
[z_1,z_2,z_3,z_4]=
\frac{z_1-z_2}{z_1-z_4}
\cdot
\frac{z_3-z_4}{z_3-z_2}.
$$
Note that the cross-ratio is invariant under M\"obius transformations.

\bigskip\noindent
For every~$i\in\z$, $i\ge1$,
let us enumerate the vertices of the triangles~$T_{i-1}$ and~$T_i$
as $Q_i,R_i,U_i,V_i$ so that:
\begin{enumerate}[$\bullet$]
\item
$Q_i,R_i$ are the shared vertices of $T_{i-1},T_i$;
\item
$U_i$ is the vertex of $T_{i-1}$ distinct from~$Q_i,R_i$;
\item
$V_i$ is the vertex of $T_i$ distinct from~$Q_i,R_i$.
\end{enumerate}
Note that $S_{i-1}\cap S_i=\{Q_i,R_i\}$.



\begin{proposition}
\label{finiteCR}
The set of complex cross-ratios of points~$Q_i,R_i,U_i,V_i$ for~$i\in\z$, $i\ge1$, is finite and does not contain zero.
\end{proposition}

\begin{proof}
For any $i\in\z$, $i\ge 1$,
the triangles $T_{i-1},T_i$ are conformally equivalent to $T_0,T_1$,
and the set of points $\{Q_i,R_i,U_i,V_i\}$ is conformally equivalent
to one of the sets
\[
  \{\Am,\An,\Ap,\tp(\An)\},\quad
  \{\Ap,\An,\Am,\tm(\An)\},
\]
hence there are only finitely many possible distinct values for their complex cross-ratios.
The points $Q_i,R_i,U_i,V_i$ are distinct by construction,
hence their cross-ratios are not equal to zero.
\end{proof}

\begin{proposition}
\label{OnePoint}
The intersection $\Deven\cap\Dodd$ is a one-point set.
\end{proposition}

\begin{proof}
Corollary~\ref{LemmaLimitCircles}$(iii)$ implies that $\Seven\cap\Sodd$ consists of one or two points.
Assume that the set~$\Deven\cap\Dodd$ consists of more than one point, then the intersection~$\Seven\cap\Sodd$ consists of exactly two distinct points.
Let $d>0$ be the Euclidean distance between these two points.
Then for every $\ve>0$ there exists $N$ such that for all~$i>N$ 
all four points $Q_i,R_i,U_i,V_i$ are in the $\ve$-neighbourhood of~$\Seven\cap\Sodd$.

Let us consider $\ve<d/3$.
Then the $\ve$-neighbourhood of the two-point set~$\Seven\cap\Sodd$
is the union of two disjoint $\ve$-discs~$W_1$ and~$W_2$.
Note that $S_{i-1}\cap S_i=\{Q_i,R_i\}$,
hence for sufficiently large~$i$,
the points~$Q_i$ and~$R_i$ are in different $\ve$-discs~$W_j$.

Suppose that the points $U_i,V_i$ are in different $\ve$-discs~$W_j$,
say $Q_i,U_i$ are in one disc and $R_i,V_i$ in the other.
Then
\[
  \Big|[Q_i,U_i,R_i,V_i]\Big|
  =\frac{|Q_i-U_i|\cdot|R_i-V_i|}{|Q_i-V_i|\cdot|U_i-R_i|}
  \le\frac{(2\ve)\cdot(2\ve)}{(d-2\ve)\cdot(d-2\ve)}
  =\frac{4\ve^2}{(d-2\ve)^2}.
\]
In the case of $Q_i,V_i$ being in one disc and $R_i,U_i$ in the other we get the same upper bound for the cross-ratio~$\Big|[Q_i,V_i,R_i,U_i]\Big|$.

If the points $U_i,V_i$ are in different $\ve$-discs~$W_j$
for infinitely many values of~$i$,
then for any~$\de>0$ we can find~$i$ such that the absolute value of the complex cross-ratio of some permutation of~$Q_i,R_i,U_i,V_i$ is smaller than~$\de$.
On the other hand, by Proposition~\ref{finiteCR}, the set of such cross-ratios is discrete and does not contain zero.
We arrive at a contradiction.

Therefore from some point in the sequence the points~$U_i,V_i$ will always be in the same $\ve$-disc~$W_j$.
It is important to note that in one step we change one of four points~$Q_i,R_i,U_i,V_i$ and keep the other three (though the labels $Q,R,U,V$ might change).

\vspace{1mm}

Suppose that $Q_i,U_i,V_i$ are in one $\ve$-disc
and $R_i$ in the other.
The triangles~$T_{i-1}$ and~$T_{i+1}$ share exactly one vertex, either~$Q_i$ or~$R_i$.
The shared vertices~$\{Q_{i+1},R_{i+1}\}$ of~$T_i$ and~$T_{i+1}$ are either~$\{Q_i,V_i\}$ or~$\{R_i,V_i\}$.
We know that $Q_{i+1}$ and~$R_{i+1}$ must be in different $\ve$-discs~$W_j$.
By assumption, $Q_i$ and~$V_i$ are in the same $\ve$-disc,
hence we must have $\{Q_{i+1},R_{i+1}\}=\{R_i,V_i\}$,
the shared vertex of~$T_{i-1}$ and~$T_{i+1}$ is~$R_i$,
and the vertices $Q_i,U_i,V_i,Q_{i+1},U_{i+1},V_{i+1}$ are all in the same $\ve$-disc. 
Similarly, if the vertices~$R_i,U_i,V_i$ are in the same $\ve$-disc,
then we must have $\{Q_{i+1},R_{i+1}\}=\{Q_i,V_i\}$,
the shared vertex of~$T_{i-1}$ and~$T_{i+1}$ is~$Q_i$
and the points $R_i,U_i,V_i,R_{i+1},U_{i+1},V_{i+1}$ are all in the same $\ve$-disc.
In either case, the binary sequence $(x_i)_{i=1}^\infty$ becomes constant from some point onwards.

\vspace{1mm}

Due to symmetry and conformal invariance, it is sufficient to consider the binary sequence~$(x_i)_{i=1}^\infty$ with~$x_i=1$ for all~$i$.
In this case the points~$U_i,V_i$ are either both on the edge~$\sn$ or both on the edge~$\sp$.
Therefore their limit is the single point~$\sn\cap\sp$.
Hence all vertices converge to a single point.
Therefore, the intersection $\Deven\cap\Dodd$ must consist of only one point.
\end{proof}

\subsubsection{Conclusion of the Proof of Theorem~\ref{convergence-triangles}:}
Corollary~\ref{corollary-nested}(\ref{part-nested})
states that the sequence of triangles~$T_i=\trioB(x_1,\dots,x_i)$ is nested, i.e.\ $T_i\subset T_{i-1}$.
Corollary~\ref{corollary-nested}(\ref{part-T-in-D})
implies that $T_i\subset D_i$ and $T_{i-1}\subset D_{i-1}$.
Taking into account $T_i\subset T_{i-1}$
we obtain $T_i\subset D_{i-1}\cap D_i$ and hence
\[
  \bigcap_{i=1}^{\infty}T_i
  \subset 
  \bigcap_{i=1}^{\infty}(D_{i}\cap D_{i-1})
  =\Deven\cap\Dodd.
\]
Proposition~\ref{OnePoint} implies that~$\Deven\cap\Dodd$
and therefore $\bigcap\limits_{i=1}^{\infty}T_i$
is a one-point set.\qed

\begin{remark}
Note that the proof of Theorem~\ref{convergence-triangles} remains valid for canonical bryophylla that correspond to the boundary line $PL$ of the curvilinear triangle $APL$ in the configuration space in Figure~\ref{pic1}.
\end{remark}

\section{Open Questions}
\label{Open questions for further study}

\noindent
Let us conclude with several open questions for further study.
In this paper we mostly study Farey bryophylla,
so the following question is of interest:

\begin{problem}
Study the properties of canonical bryophylla that correspond to some regions in the complement of the curvilinear triangle $APL$ in the configuration space in Figure~\ref{pic1}.
\end{problem}

\bigskip\noindent
In our study of complex bryophylla in~$\partial\h^3$, we make use of the identification with the complex plane and techniques available in the one-dimensional situation.
In $\partial\h^4$,
the situation might be substantially different,
so we can pose the following problem:

\begin{problem}
Generalise the notion of conformal bryophylla to higher dimensions.
\end{problem}

\bigskip\noindent
There is hope that the calculations using complex numbers can be replaced by calculations with some conformal invariants such as angles between circles, cross-ratios of four points etc. 
The first step in this direction might be the following question:

\begin{problem}
Link the Gauss measure on conformal bryophylla
with conformal invariants of the set.
\end{problem}

\bigskip\noindent
Currently the Gauss measure is defined as a lifting from the segment~$[0,1]$, which does not immediately relate to its conformal properties. 
Note that a similar approach was used to study Gauss-Kuzmin statistics for multi-dimensional continued fractions in the sense of Klein (see e.g.~\cite{Kontsevich1999}
and~\cite{Karpenkov2007}).


\providecommand{\bysame}{\leavevmode\hbox to3em{\hrulefill}\thinspace}
\providecommand{\MR}{\relax\ifhmode\unskip\space\fi MR }
\providecommand{\MRhref}[2]{%
  \href{http://www.ams.org/mathscinet-getitem?mr=#1}{#2}
}
\providecommand{\href}[2]{#2}

\end{document}